\documentclass[12pt]{scrartcl}
 \usepackage{cmap}
 \usepackage{amsmath,amsxtra,amscd,amssymb,latexsym,stmaryrd,amsthm, mathtools}
  \usepackage{xcolor,mathrsfs}

\usepackage{lmodern} 
\usepackage[english]{babel}
\usepackage[utf8]{inputenc}
\usepackage[T1]{fontenc}

\usepackage[colorlinks=true,linkcolor=red!70!black,citecolor=green!70!black, urlcolor=magenta!70!black,backref]{hyperref}

\addtokomafont{title}{\rmfamily\itshape}

\theoremstyle{plain}
\newtheorem{theorem}{Theorem}[section]
\newtheorem*{theo*}{Theorem}
\newtheorem{corollary}[theorem]{Corollary}
\newtheorem{proposition}[theorem]{Proposition}
\newtheorem{lemma}[theorem]{Lemma}
\newtheorem{theomain}{Theorem}

\newtheorem*{assumption}{Standing Assumption}

\theoremstyle{definition}
\newtheorem{definition}[theorem]{Definition}
\newtheorem{remark}[theorem]{Remark}
\newtheorem{example}[theorem]{Example}

\newcommand*{\dd}%
  {\relax\ifnum\lastnodetype>0\mskip\medmuskip\fi\mathrm{d}}

\newcommand{\proba}{\operatorname{\mathcal{P}}}
\newcommand{\wass}{\operatorname{W}}
\newcommand{\diam}{\operatorname{diam}}
\newcommand{\basin}{\operatorname{Ba}}
\newcommand{\supp}{\operatorname{supp}}
\newcommand{\stabL}{\operatorname{Sl}}
\newcommand{\Id}{\mathrm{Id}}
\newcommand{\po}{\operatorname{p\omega}}
\newcommand{\op}[1]{\mathscr{#1}}
\newcommand{\chop}[1]{\check{\mathscr{#1}}\rlap{\phantom{#1}}}
\newcommand{\corr}[2]{\operatorname{\mathcal{C}}^{#1}_{#2}}

\newcommand{\etilde}{%
  \mskip0.5\thinmuskip\tilde\null\mskip\thinmuskip}
\newcommand{\one}{\boldsymbol{1}}
\newcommand{\Hol}{\operatorname{Hol}}
\newcommand{\hol}[2][]{\pmb{\omega}_{#2}^{#1}}
\newcommand{\holl}[1]{\pmb{\omega}_{#1\log}}
\newcommand{\C}{\mathscr{C}}
\newcommand{\alg}[1]{\mathscr{#1}}
\newcommand{\fspace}[1]{\operatorname{\mathcal{#1}}}
\newcommand{\Banach}{\fspace{B}}
\newcommand{\ent}[1][]{h_{#1}}
\newcommand{\entKS}{\ent[\mathrm{KS}]}
\newcommand{\LimTh}[1]{\mathscr{#1}}
\newcommand{\UE}{\operatorname{UE}}
\newcommand{\FreeE}{\operatorname{\mathscr{F}}}

\newcommand*{\dvol}%
  {\relax\ifnum\lastnodetype>0\mskip\medmuskip\fi\mathrm{dVol}}

\title{Extensions with shrinking fibers}

\author{Beno\^{\i}t R. Kloeckner \thanks{LAMA, Univ Paris Est Creteil, Univ Gustave Eiffel, UPEM, CNRS, F-94010, Cr\'eteil, France}}

\begin{document}

\maketitle

\begin{abstract}
We consider dynamical systems $T:X\to X$ that are extensions of a factor $S:Y\to Y$ through a projection $\pi: X\to Y$ with \emph{shrinking} fibers, i.e. such that $T$ is uniformly continuous along fibers $\pi^{-1}(y)$ and the diameter of iterate images of fibers $T^n(\pi^{-1}(y))$ uniformly go to zero as $n\to\infty$.

We prove that every $S$-invariant measure $\check\mu$ has a unique $T$-invariant lift $\mu$, and prove that many properties of $\check\mu$ lift to $\mu$: ergodicity, weak and strong mixing, decay of correlations and statistical properties (possibly with weakening in the rates).

The basic tool is a variation of the Wasserstein distance, obtained by constraining the optimal transportation paradigm to displacements along the fibers. We extend to a general setting classical arguments, enabling to translate potentials and observables  back and forth between $X$ and $Y$.
\end{abstract}

\tableofcontents

\section{Introduction}

Let $T:X\to X$ be a dynamical system where $X$ is a compact metric space, and assume that $T$ has a topological factor $S:Y\to Y$, i.e. there is a continuous onto map $\pi :X\to Y$ such that $\pi T = S\pi$. Each fiber $\pi^{-1}(y)\subset X$ is collapsed under $\pi$ into a single point $y$, and $S$ can thus be thought of as a simplification of $T$, which may retain certain of its dynamical properties but forget others. When additionally $T$ shrinks the fibers, i.e. two points $x,x'$ such that $\pi(x)=\pi(x')$ have orbits that are attracted one to another, one suspects that actually \emph{all} important dynamical features of $T$ survive in $S$: along the fibers, the dynamic is trivial anyway. It might still happen that $S$ is easier to study than $T$, in which case one can hope to obtain interesting dynamical properties of $T$ by proving them for $S$ and lifting them back. The present article aims at developing a systematic machinery to do that in the context of the thermodynamical formalism, i.e. the study of equilibrium states (invariant measure optimizing a linear combination of entropy and energy with respect to a potential).

This setting has already been largely studied, first in the symbolic case $X=\{0,1\}^{\mathbb{Z}}$ (or a subshift), $Y=\{0,1\}^{\mathbb{N}}$ (or the corresponding one-sided subshift), $\pi$ the map that forgets negative indexes, $T$ and $S$ the left shifts. Then the strategy outlined above has been used for long, see e.g. \cite{bowen1975equilibrium}. The advantage of one-sided shift is that an orbit can be looked backward in time as a non-trivial, contracting Markov Chain; one can use this to prove existence, uniqueness and statistical properties of equilibrium states for a wide range of potentials. The same reason makes expanding maps quite easier to study than hyperbolic ones. Recently, several works have used the above approach to study various flavor of hyperbolic dynamics on manifolds or domains of $\mathbb{R}^n$. However they are often written for specific systems and the technical details are often not obviously generalizable. Moreover, the basic result that an $S$-invariant measure of $S$ has a \emph{unique} $T$-invariant lift seems not to be known in general. Our first aim will be to propose a simple and general argument, based on ideas from optimal transportation, to lift invariant measures and show uniqueness. Then we shall use uniqueness and adapt folklore methods to a general framework to lift a rather complete set of properties of invariant measures.

While the dynamical study of uniformly hyperbolic maps is considered reasonably well understood, the study of various kind of non-uniformly hyperbolic maps has witnessed a large activity in the last two decades, see e.g. \cite{young98,alves2000SRB,araujo2015rapid,alves2017SRB}, the surveys \cite{alves2015srb,crovisier2015introduction} and other references cited below. Even in the uniformly expanding case, new approaches are welcome, see \cite{climenhaga2018equilibrium}.
As is well-known, the ``extension'' approach can be used to study certain uniformly and non-uniformly hyperbolic maps, when the default of hyperbolicity can be in the contraction or the expansion, or when the potential lacks H\"older  regularity (see Section \ref{sec:corollaries} and Remark \ref{rema:hyperbolic}).

\subsection{Shrinking fibers: main definition and first main result}

All measures considered are probability measures, and we denote by $\proba(X)$ the set of measures on $X$ and by $\proba_T(X)$ the set of $T$-invariant measures, both endowed with the weak-$*$ topology.

We shall consider the case when the extension $T$ exhibits some contraction along fibers; we introduce a single notion that includes a global property and continuity along fibers (details on moduli of continuity are recalled in Section \ref{sec:def-moduli}).
\begin{definition}\label{defi:shrinking}
We say that $T$ is \emph{uniformly continuous} along fibers whenever there a modulus of continuity $\bar\omega_T$ such that for all $x,x'\in X$ with $\pi(x)=\pi(x')$,
\[d(Tx,Tx') \le \bar\omega_T(d(x,x')).\]

We say that $T$ is an extension of $S$ with \emph{shrinking fibers} (keeping implicit $\pi$ and $S$) whenever $T$ is uniformly continuous along fibers and there is a sequence $(a_n)_{n\in\mathbb{N}}$ of positive numbers such that $\lim a_n = 0$ and for all $x,x'\in X$ with $\pi(x)=\pi(x')$:
\[d(T^n x,T^n x') \le a_n.\]
If $a_n=C\theta^n$ for some $\theta\in(0,1)$, (respectively: $a_n=Cn^{-d}$ for some $d>0$) we may specify that $T$ has \emph{exponentially} (respectively \emph{polynomially}) shrinking fibers, of \emph{ratio $\theta$} (respectively \emph{degree $d$}). We may specify the \emph{shrinking sequence} $(a_n)_n$, e.g. by saying that $T$ has $(a_n)_n$-shrinking fibers.
\end{definition}

For example, if for some $\theta\in(0,1)$ and all $x,x'\in X$ such that $\pi(x)=\pi(x')$ we have $d(Tx, Tx') \le \theta d(x,x')$, then $T$ has exponentially shrinking fibers; however, the latter property is weaker.

Of course, if $T$ is continuous then it is uniformly continuous along fibers;  the first part of the above definition is meant to make it possible to deal quite generally with discontinuous maps $S$. It shall be used mainly in the proof of Theorem \ref{theo:lift}, which is at the core of Theorem \ref{theo:mainLift} below.

Some of our results actually hold more generally, and to state them in their full scope we introduce the following notion.
\begin{definition}\label{defi:shrinking-average}
Given $\check\mu\in\proba_S(Y)$, we say that $T$ is an extension of $S$ whose \emph{fibers are shrunk on average with respect to $\check\mu$} whenever $T$ is uniformly continuous along fibers and there is a sequence $(\bar a_n)_{n\in\mathbb{N}}$ of positive numbers such that $\lim_n \bar a_n = 0$ and
\[ \int \diam\big(T^n(\pi^{-1}(y))\big) \dd\check\mu(y) \le \bar a_n \quad\forall n\in\mathbb{N}.\]
\end{definition}

At some points we will also need some mild additional regularity for $\pi$.
\begin{definition}
We say that the projection $\pi$ is \emph{non-singular} (with respect to $\lambda_X$ and $\lambda_Y$) when $\pi_*\lambda_X$ is equivalent to $\lambda_Y$, i.e. when for all Borel set $B\subset Y$:
\[\lambda_Y(B) > 0 \Leftrightarrow \lambda_X(\pi^{-1}(B)) > 0.\]
\end{definition}

\begin{definition}\label{defi:fibration}
We say that $\pi$ \emph{induces a continuous fibration} whenever for some modulus of continuity $\bar\omega_\pi$ and all $y,y'\in Y$, for all $x\in\pi^{-1}(y)$ there exist $x'\in\pi^{-1}(y')$ such that
\[d(x,x') \le \bar\omega_\pi(d(y,y')).\]

Then by the Measurable Selection Theorem, there exist a measurable map $\tau_y^{y'} : \pi^{-1}(y) \to \pi^{-1}(y')$ such that $d(x,\tau_y^{y'}(x)) \le \bar\omega_\pi(d(y,y'))$ for all $x\in\pi^{-1}(y)$.
\end{definition}
The modulus of continuity $\bar\omega_\pi$ shall not be confused with the modulus of continuity of the map $\pi$, which whenever needed shall be denoted by $\omega_\pi$.\\

With these definitions set up, we can state our first results gathered in the following statement. 
\begin{theomain}\label{theo:mainLift}
If $T$ is an extension of $S$ with shrinking fibers, then $\pi_*$ induces a homeomorphism from $\proba_T(X)$ to $\proba_S(Y)$. 
In particular, for each $S$-invariant measure $\check\mu$ there is a unique $T$-invariant measure $\mu$ such that $\pi_*\mu=\check\mu$. 

Moreover for all (non necessarily invariant) $\nu\in\proba(X)$ such that $\pi_*\nu=\check\mu$, we have $T^n_*\nu \to \mu$ in the weak-$*$ topology, and:
\begin{enumerate}
\item\label{enumi:theomainlift1} each of the following adjectives applies to $\mu$ (with respect to $T$) if and only if it applies to $\check\mu$ (with respect to $S$):
\emph{ergodic}, \emph{weakly mixing}, \emph{strongly mixing},
\item\label{enumi:theomainlift2} if $T$ is continuous and $X,Y$ have reference measures with respect to which $\pi$ is non-singular, then each of the following adjectives applies to $\mu$ if and only if it applies to $\check\mu$ :
\emph{physical}, \emph{observable},
\item\label{enumi:theomainlift3} if $S,T$ are continuous, then $\mu$ and $\check\mu$ have the same Kolmogorov-Sinai entropy.
\end{enumerate}
\end{theomain}

Theorem \ref{theo:mainLift} is unsurprising, and some parts are already known in more or less general settings (e.g. lifting of physicality); however the uniqueness of the lift was not known in general, and simplifies a lot the proof of further properties. 
It is in particular interesting to compare the existence and uniqueness part of Theorem \ref{theo:mainLift} to Section 6.1 of \cite{araujo2009singular} where Araujo, Pacifico, Pujals and Viana construct a lift of $\check\mu$. The first advantage of our result is that we prove uniqueness among all invariant measures, while they get uniqueness only under the property they use in the construction. Second, we need milder assumptions (see Remark \ref{rema:APPV}).

Castro and Nascimento have studied in \cite{castro2017statistical} two kinds of maps, the first one fitting in the theme of the present article. Namely, they consider the case when $S$ is a non-uniformly expanding map in the family introduced by Castro and Varandas \cite{CV} and $T$ is exponentially contracting along fibers. They focus there on the maximal entropy measure for $T$, proving it exists, is unique, and enjoys exponential decay of correlations and a Central Limit Theorem for H\"older observables. Leaving aside the statistical properties for now, Theorem \ref{theo:mainLift} in particular shows that existence and uniqueness of the maximal entropy measure for $T$ does not depend on the specifics of $S$ nor on the rate of contraction along fibers (as said, we actually do not even need $T$ to be a contraction along fibers, only to shrink them globally): item \ref{enumi:theomainlift3} is a broad generalization of Theorem A from \cite{castro2017statistical} since under  the only assumptions that $S,T$ are continuous and that fibers are shrinking, it shows that $T$ has a unique measure of maximal entropy if and only if $S$ does.


\subsection{Further main results: thermodynamical formalism}

In our subsequent results, we shall assume $T$ is Lipschitz and this hypothesis deserves an explanation. We will often need to work in some functional spaces where observables or potentials are taken, and we made the choice of \emph{generalized H\"older spaces}, i.e. spaces of function with modulus of continuity at most a multiple of some reference, arbitrary modulus. This choice seems a good balance between generality (it includes functions less regular than H\"older, enabling us to consider in particular polynomial rates of shrinking) and clarity (proofs stay pretty simple and the amount of definition needed is significant but not overwhelming). It is often a crucial ingredient that the iterated Koopman operators $f\mapsto f\circ T^k$ are bounded on the chosen functional space, with good control of their norms; asking $T$ to be Lipschitz is the natural hypothesis to ensure this for generalized H\"older spaces. Where one interested of discontinuous maps (e.g. when $S$ is discontinuous), the principle of proofs could certainly be adapted but one would need (as usual) to work in a suitable functional spaces. Another advantage of our choice is that we can work directly with the Wasserstein distance between measures.\\

The convergence result ($T^n\nu\to \mu$ whenever $\pi_*\nu=\check\mu$) in Theorem \ref{theo:mainLift} seems new in this generality. It is however not as satisfying as those obtained by Galatolo and Lucena in Section 5.1 of \cite{galatolo2015spectral} in their particular setting, where  instead of $\pi_*\nu=\check\mu$ it is only asked that $\pi_*\nu$ is absolutely continuous with respect to $\check\mu$ (with some regularity assumptions on the density). In this direction, we prove the following variation of \cite[Section 5.1]{galatolo2015spectral} (our hypotheses are quite general, but we assume $T$ to be Lipschitz and our convergence is in the Wasserstein metric instead of the particular metric constructed in \cite{galatolo2015spectral}).
\begin{theomain}\label{theo:mainSG}
Assume that $T$ is an extension of $S$ with exponentially shrinking fibers; that $S$ admits a conformal measure $\lambda_Y\in\proba(Y)$ such that the associated transfer operator has a spectral gap on some Banach space $\Banach$, whose normalized eigenfunction is denoted by $h$; that $T$ is Lipschitz and that $\pi$ induces a H\"older-continuous fibration;
and let $\mu$ be the unique $T$-invariant lift of $h\dd\lambda_Y$.

Then for all $\nu\in\proba(X)$ such that $\pi_*\nu$ is absolutely continuous with respect to $\lambda_Y$ with density in $\Banach$, the sequence $(T_*^n\nu)_n$ converges to $\mu$ exponentially fast in the Wasserstein metric.
\end{theomain}
The needed, classical definitions are given in Section \ref{sec:prelim}, in particular the Wasserstein metric is defined in Section \ref{sec:def-Wasserstein} (for now, let us simply say that in our compact setting it metrizes the weak-$*$ topology) and transfer operators are defined in Section \ref{sec:def-transfer}.
A more general (but less precise) result is given in Corollary \ref{coro:StableLeaf}.\\

We now turn to equilibrium states and their statistical properties.
It will be convenient to use the following definition (the reader may want to have a look at Section \ref{sec:def-moduli} about moduli of continuity and generalized H\"older spaces; in particular, we shall use the very mild modulus of continuity $\holl{\alpha}(r)\simeq (\log\frac1r)^{-\alpha}$).
\begin{definition}
Let $\omega_p,\omega_o$ be moduli of continuity, let $\rho\in(0,\infty]$ and let $\LimTh{T}$ be the name of a limit theorem for discrete-time random processes (e.g. $\mathrm{LIL}$ for the Law of Iterated Logarithm, $\mathrm{CLT}$ for the Central Limit Theorem, or $\mathrm{ASIP}$ for the Almost Sure Invariance Principle, see Definition \ref{defi:stat}). We shall say that $T$ \emph{has unique equilibrium states for potentials in $\Hol_{\omega_p}(X)$ of norm less than $\rho$, with limit theorem $\LimTh{T}$ for observables in $\Hol_{\omega_o}(X)$} (in short, that $T$ satisfies $\UE(\omega_p[\rho];\LimTh{T},\omega_o)$) whenever:
\begin{enumerate}
\item for all potentials $\varphi\in\Hol_{\omega_p}(X)$ such that $\lVert \varphi\rVert_{\omega_p}<\rho$, there exist a unique \emph{equilibrium state} $\mu_\varphi$, i.e. a maximizer of the free energy $\FreeE(\mu) = \entKS(T,\mu)+\mu(\varphi)$ over all $\mu\in\proba_T(X)$ (where $\entKS$ denotes Kolmogorov-Sinai entropy), and
\item for all $f\in \Hol_{\omega_o}(X)$, the random process $(f\circ T^k(Z))_{k\in\mathbb{N}}$, where $Z$ is a random variable with law $\mu_\varphi$, satisfies the limit theorem $\LimTh{T}$ (see Definition \ref{defi:stat} for more precisions).
\end{enumerate}
When there is no bound on the norm of potential, i.e. $\rho=\infty$, we may shorten $\UE(\omega_p[\infty];\LimTh{T},\omega_o)$ into $\UE(\omega_p;\LimTh{T},\omega_o)$. When this property is satisfied for all H\"older-continuous potentials or observables, whatever the H\"older exponent, we write $\hol{*}$ in place of $\omega_p$ or $\omega_o$.  When we only want to state existence and  uniqueness of equilibrium state, we agree to take $\LimTh{T}=\varnothing$ and we can simplify the notation into $\UE(\omega_p;\varnothing)$ as there is no need to specify the observables.
\end{definition}

\begin{theomain}\label{theo:mainstat}
Assume that $T$ is $L$-Lipschitz and that $\pi$ is $\beta$-H\"older and admits a Lipschitz section and let $\LimTh{T}\in\{\varnothing,\mathrm{LIL},\mathrm{CLT},\mathrm{ASIP}\}$.
\begin{enumerate}
\item\label{enumi:mainstat1} Assume that the fibers are exponentially shrinking with ratio $\theta$, let $\alpha\in(0,1]$ and set $\gamma = \frac{\alpha\beta}{1-\log L/\log \theta}$. If $S$ satisfies $\UE(\hol{\gamma}[\rho];\LimTh{T},\hol{*})$ then $T$ satisfies $\UE(\hol{\alpha}[C\rho];\LimTh{T},\hol{*})$ for some $C>0$.
\item\label{enumi:mainstat2} Assume that the fibers are polynomially shrinking with degree $d>1$, consider $\alpha,\gamma \in (1/d,1]$ and set $\alpha'=\alpha d-1$ and $\gamma'=\gamma d-1$. If $S$ satisfies $\UE(\holl{\alpha'}[\rho];\LimTh{T},\holl{\gamma'})$, then $T$ satisfies $\UE(\hol{\alpha} [C\rho];\LimTh{T},\hol{\gamma})$ for some $C>0$.
\item\label{enumi:mainstat3} Assume that the fibers are exponentially shrinking, consider $\alpha,\gamma>1$ and let $\alpha'=(\alpha-1)/2$ and $\gamma'=(\gamma-1)/2$. If $S$ satisfies $\UE(\holl{\alpha'}[\rho];\LimTh{T}, \holl{\gamma'})$ then $T$ satisfies  $\UE(\holl{\alpha}[C\rho];\LimTh{T}, \holl{\gamma})$ for some $C>0$.
\end{enumerate}
\end{theomain}
Many other combinations of moduli of continuity and shrinking speed can be considered, see Theorem \ref{theo:stat}. The main tool is to construct from a potential or an observable on $X$ a suitable potential or observable on $Y$. For this, we generalize a method that is classical in the symbolic setting: adding a coboundary to make the potential or observable constant along fibers.

Item \ref{enumi:mainstat1} generalizes Theorems A and C of \cite{castro2017statistical}: Castro and Nascimento where concerned with the maximal entropy measure, i.e. the equilibrium state for the null potential, while item \ref{enumi:mainstat1} provides in their setting (using the known results for $S$ a Castro-Varandas maps \cite{CV}) existence, uniqueness and CLT for H\"older observables for the equilibrium state of any H\"older potential of small enough norm. The generalization is actually far broader, since one can take much more varied base maps $S$ for which equilibrium states have the desired limit theorem (see e.g. \cite{melbourne2002central,MN05, gouezel2010almost} for the ASIP). Some examples will be provided below.

Interestingly, it appears more efficient to directly lift limit theorems from $S$ to $T$ than to lift decay of correlations and then use them to prove limit theorem for $T$. Nevertheless, decay of correlations have a long history and are prominent features of invariant measures, and it thus makes sense to lift them as well. In this regard, we obtain the following result.
\begin{theomain}\label{theo:maindecay}
Let $T$ be a Lipschitz extension of $S$ with shrinking fibers, assume that there is a Lipschitz section $\sigma$, let $\mu$ be $T$-invariant probability measure and $\check\mu := \pi_*\mu$ be the corresponding $S$-invariant measure.
\begin{enumerate}
\item\label{enumi:decay1} If for some $\alpha_0\in(0,1]$ the transfer operator $\chop{L}$ of $(S,\check\mu)$ has a spectral gap in each H\"older space $\Hol_\alpha(Y)$ ($\alpha\in(0,\alpha_0]$) and if fibers are exponentially shrinking, then $\mu$ has exponential decay of correlation in each H\"older space.
\item\label{enumi:decay2} If for some $\alpha\in (0,1]$ and all $n\in\mathbb{N}$ the transfer operator $\chop{L}$ of $(S,\check\mu)$ satisfies
\[ \Hol_{\alpha}(\chop{L}^n h) \lesssim \Hol_{\alpha}(h) \quad\text{and}\quad \lVert \chop{L}^n f \rVert_\infty \lesssim \frac{\Hol_{\alpha d \log}(f)}{n^p} \]
whenever $h\in\Hol_{\alpha}(Y)$, $f\in\Hol_{\alpha d\log}(Y)$ with $\check\mu(f)=0$,
and if the fibers of $\pi$ are polynomially shrinking of degree $d$, then $\mu$ has polynomial decay of correlation of degree $\min(\alpha d,\frac{p}{2})$ for all $\hol{\alpha}$-continuous observables.
\item\label{enumi:decay3} If for some $\alpha >0$ and all $n\in\mathbb{N}$ the transfer operator $\chop{L}$ of $(S,\check\mu)$ satisfies
\[ \Hol_{\alpha \log}(\chop{L}^n h) \lesssim \Hol_{\alpha \log}(h) \quad\text{and}\quad \lVert \chop{L}^n f \rVert_\infty \lesssim \frac{\Hol_{\frac\alpha2 \log}(f)}{n^p} \]
whenever $h\in\Hol_{\alpha\log}(Y)$, $f\in\Hol_{\frac\alpha2\log}(Y)$ with $\check\mu(f)=0$,
and if the fibers of $\pi$ are exponentially shrinking, then $\mu$ has polynomial decay of correlation of degree $\min(\alpha,\frac{p}{2})$ for all $\holl{\alpha}$-continuous observables.
\end{enumerate}
\end{theomain}
To prove this result, our main tool is expected: we prove the regularity of the disintegration of $\mu$ with respect to $\pi$ (Theorem \ref{theo:regularity}). Such results appeared in the work of Galatolo and Pacifico \cite{galatolo2010lorenz} (Appendix A; extra difficulty in the proof there seems to be caused by the way disintegration is set up, making it necessary to deal with non-probability measures) and in the recent works of Butterley and Melbourne \cite{butterley2017disintegration} (Proposition 6, to compare with Theorem \ref{theo:regularity}) and of Araujo, Galatolo and Pacifico \cite{araujo2014decay} (Theorem A). Compared to these work, we gain in generality: we can consider very general maps while they tend to restrict to uniformly expanding maps, we consider an arbitrary $S$-invariant measure instead of restricting to the absolutely continuous one. Items \ref{enumi:decay2} and \ref{enumi:decay3} have no equivalent that I know of in the literature.

\subsection{A few examples}\label{sec:corollaries}

A commonly studied situation where our framework applies readily is that of skew-products, where $X=Y\times \Phi$ for some compact metric space $\Phi$ and
\[ T : x=(y,\phi) \mapsto (S(y),R(y,\phi)).\]
The fact that $T$ shrinks fibers then translates into
$d(R^n(y,\phi),R^n(y,\phi'))\le a_n$ for all $n\in\mathbb{N}$ and all $\phi,\phi'\in\Phi$ where $a_n\to 0$, $R^1=R$ and $R^{n+1}(y,\phi) = R(S(y),R^n(y,\phi))$. The projection map is then $\pi:(y,\phi)\to y$ and all needed hypotheses on $\pi$ in Theorem \ref{theo:mainLift}-\ref{theo:mainstat} are easy to check, endowing for example $X$ with the metric $d((y,\phi),(y',\phi'))=\sqrt{d(y,y')^2+d(\phi,\phi')^2}$. Note that since we will apply our above results, in many cases we will assume $T$ (and thus $S$ and $R$) to be Lipschitz; and in all cases our ``shrinking fiber'' hypothesis implies that $R(y,\phi)$ depends continuously on the variable $\phi$ when $y$ is fixed.
\begin{remark}
One can easily generalize this setting to fiber bundles: $X$ is then no longer a product, but there is a compact metric space $\Phi$ and a \emph{fibered atlas} $(U_i,h_i)_{i\in I}$, i.e. the $U_i$ form an open cover of $Y$ and the $h_i$ are homeomorphisms from $U_i\times \Phi$ to $\pi^{-1}(U_i)$ such that $h_i(y\times\Phi)=\pi^{-1}(y)$. The simplest example of a fiber bundle that is not a product is the M\"obius band, together with the usual projection on the circle. In this setting, $T$ is asked to send fibers into fibers and is \emph{locally} of the form $(y,\phi) \mapsto (S(y),R(y,\phi))$ (where the charts $h_i$ are used to identify $\pi^{-1}(U_i)$ with a product). Our main results are stated in an even more general framework, and in order to aim for simplicity we shall restrict the examples to product spaces, but fiber bundles seem unjustly under-represented in the dynamical literature.
\end{remark}

We shall consider examples spanning all the following weaknesses of the system to be considered:  non-uniform expansion in the ``horizontal'' direction, slow shrinking in the ``vertical'' direction, or low-regularity of potentials (and observables).

Let us recall the classical benchmark for non-uniformly expanding maps, the Pomeau-Manneville family defined on the circle $\mathbb{T}=\mathbb{R}/\mathbb{Z}$ (identified with $[0,1)$) by
\begin{align*}
S_q : \mathbb{T} &\to \mathbb{T} \\
  y &\mapsto \begin{cases} 
    \big(1+(2y)^q \big)y &\text{if }y\in [0,\frac12] \\ 
    2y-1 &\text{if } y\in [\frac12, 1)
             \end{cases}
\end{align*}
where $q\ge 0$ (when $q=0$ we get the doubling map, which is \emph{uniformly} expanding).
Let $\Phi$ be a compact metric space endowed with a reference (finite, positive) measure $\lambda_\Phi$, denote by $\lambda_\mathbb{T}$ the Lebesgue measure on the circle, and endow $X=\mathbb{T}\times\Phi$ with the reference measure $\lambda_X := \lambda_\Phi\times \lambda_{\mathbb{T}}$.
\begin{corollary}\label{coro:mainPM}
Assume $T:\mathbb{T}\times \Phi\to \mathbb{T}\times \Phi$ is a continuous skew product with base map $S_q$ for some $q\in(0,1)$ and with shrinking fibers. 
\begin{enumerate}
\item\label{enumi:PM2} $T$ admits a unique physical measure $\mu_T$.
\item Assume further that $q<\frac12$, that fibers are exponentially shrinking and that $T$ is Lipschitz. Then $\mu_T$ satisfies the ASIP for all H\"older-continuous observables.
\item\label{enumi:PM1} If $T$ is $L$-Lipschitz and fibers are exponentially shrinking with ratio $\theta$ and if $q':=q(1- \log L/\log\theta)$ is less than $1$, then for all $\alpha\in(q',1]$ the map $T$ satisfies $\UE(\hol{\alpha};\mathrm{CLT},\hol{*})$; i.e. each $\alpha$-H\"older potential $\varphi$ has a unique equilibrium state $\mu_\varphi$, for which H\"older-continuous observables satisfy the Central Limit Theorem. Moreover H\"older observables have exponentially decaying correlations with respect to $\mu_\varphi$.
\end{enumerate}
\end{corollary}

Let us say that a map $S:Y\to Y$ is \emph{uniformly expanding} when it is a self-covering map of degree $k$, and there is some $\theta\in (0,1)$ such that for each $y,y'\in Y$, denoting by $z_1,\dots,z_k$ and $z'_1,\dots, z'_k$ the inverse images of $y$ and $y'$, there is a permutation $\sigma$ such that $d(z_i,z'_{\sigma(i)})\le \theta d(y,y')$ for all $i$. We assume here for simplicity that $Y$ is a manifold endowed with a volume form yielding a reference measure $\lambda_Y$ and that $X=Y\times \Phi$ is again endowed with a product measure $\lambda_X=\lambda_Y\times\lambda_\Phi$.
\begin{corollary}\label{coro:mainpoly}
Assume that $T$ is a Lipschitz skew-product with base map $S$ a uniformly expanding map and with polynomially shrinking fibers of degree $d>2$, and let $\alpha\in(2/d,1]$. Then for all $\alpha$-H\"older potential $\varphi$:
\begin{enumerate}
\item $T$ has a unique equilibrium state $\mu_\varphi$,
\item if $\alpha>\frac{5}{2d}$, then $T$ satisfies $\UE(\hol{\alpha};\mathrm{CLT},\hol{\alpha-\frac1d})$: we have the Central Limit Theorem for $\mu_\varphi$ and for all $(\alpha-\frac1d)$-H\"older observables.
\item for all $\gamma\in [\alpha-1/d,1]$, all $\gamma$-H\"older observable have polynomial decay of correlations of degree $\frac{\alpha d}{2}-1$ with respect to $\mu_\varphi$,
\end{enumerate}
\end{corollary}
Note that the second item is not enough to obtain the third one: when $\alpha$ is only slightly above $5/2d$, we get decay of correlation of degree only slightly above $1/4$ while degree $1/2$ would be a minimum to obtain the CLT. This is a sign that Theorem \ref{theo:maindecay} might not be optimal.

Let us now consider low-regularity maps, i.e. below the $\C^{1,\alpha}$ regularity.
\begin{corollary}\label{coro:mainlog}
Assume that $N$ is a manifold admitting a uniformly expanding $\C^{1,\alpha\log}$ map $S:N\to N$ with expanding factor $\lambda$, that $D$ is a $d$-dimensional closed ball with some Riemannian metric, and that $T:M =N\times D \to M$ writes as a continuous skew-product $T(y,z)=(S(y),R(y,z))$ with shrinking fibers.

Then $T$ admits a unique physical measure $\mu_T$, and the basin of attraction of $\mu_T$ has full volume.

If $T$ is Lipschitz and has exponentially shrinking fibers (in particular, when $T$ is uniformly hyperbolic), then $\mu_T$ has a polynomial decay of correlations of degree $(\alpha-1)/2$ for $\C^{(2\alpha-2)\log}$ observables. If moreover $\alpha>3/2$, then $\mu_T$ satisfies the Central Limit Theorem for $\C^{(2\alpha-1)\log}$ observables.
\end{corollary}
The regularities needed on observables are quite weak (they include in particular all H\"older observables) but the assumption that $T$ is a skew-product is very strong; it is a whole research project to consider the case when $T$ is a general $\C^{1+\alpha\log}$ uniformly hyperbolic diffeomorphism onto its image, e.g. with a ``solenoidal'' attractor: one can quotient out by the stable foliation, obtaining a skew-product over a bundle, but only up to a conjugacy as regular as the foliation. Often, this conjugacy is \emph{not} $\C^1$, and the regularity of the foliation's holonomy needs to be finely controlled to overcome this difficulty.

\begin{remark}
The skew-products of the above corollaries need not be diffeomorphisms, and can have intricate attractors: Figure \ref{fig:attractors} shows some examples with $Y=\mathbb{T}$ and $\Phi=[0,1]$, and with $S$ an expanding map of the circle.
\end{remark}

\begin{figure}[htp]
\begin{center}
\includegraphics[width=.33\linewidth]{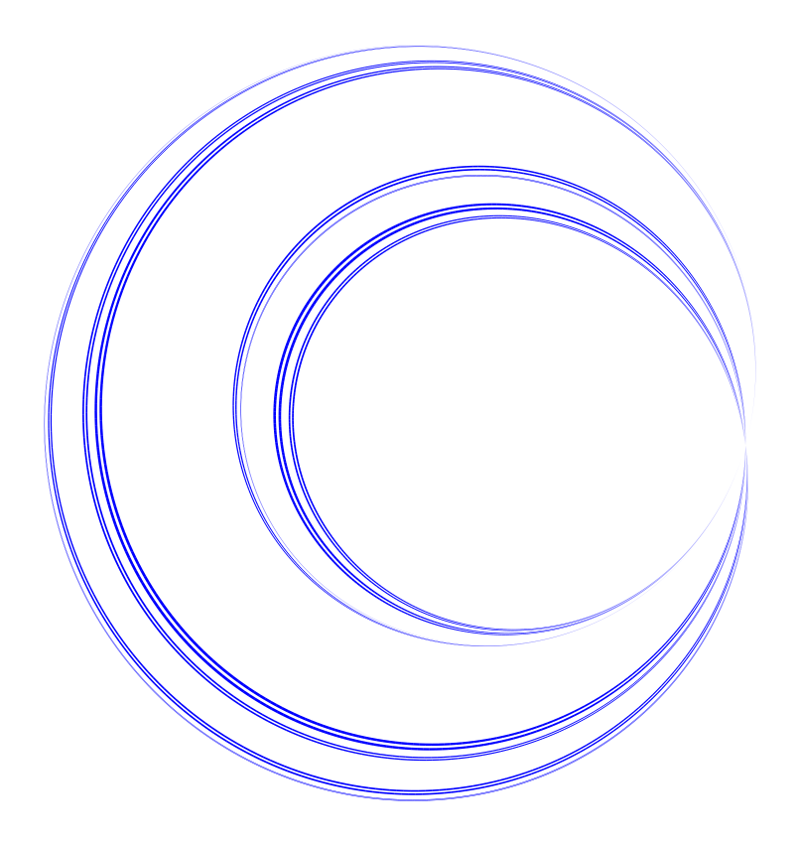}\hspace{\stretch{1}}%
\includegraphics[width=.33\linewidth]{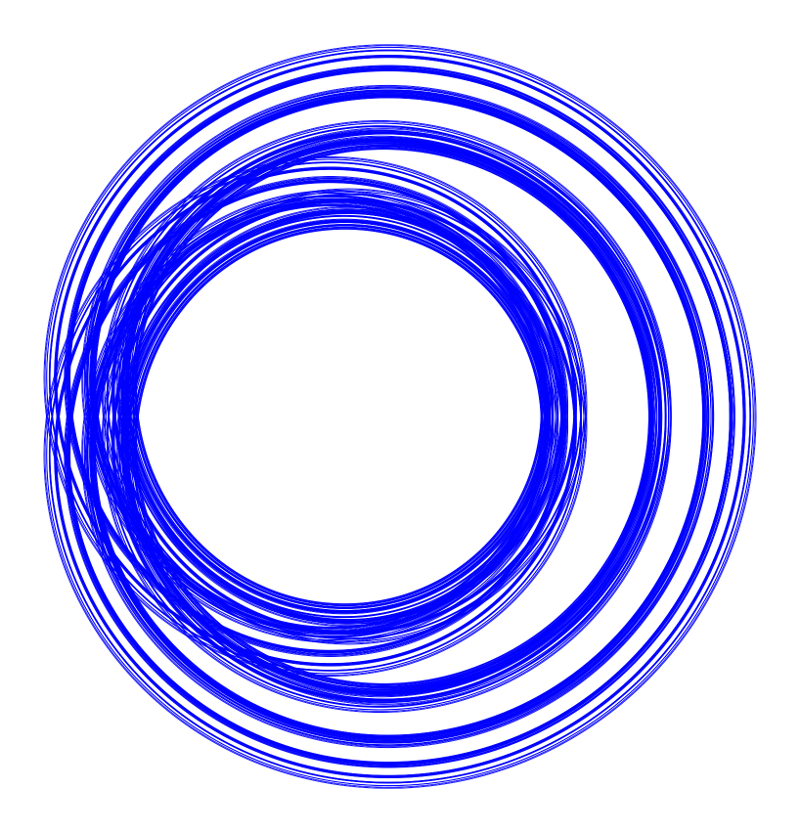}\hspace{\stretch{1}}%
\includegraphics[width=.33\linewidth]{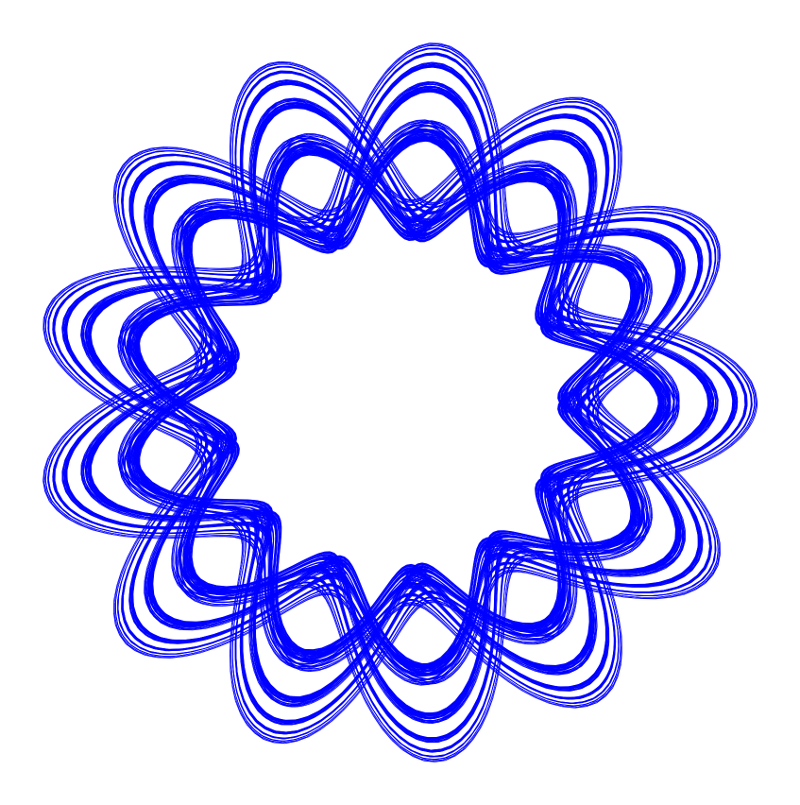}
\caption{Some attractors on the annulus: the images picture $T_i^{n_i}(U)$ with $U$ a neighborhood of the attractor, $n_1=4$, $n_2=9$ and $n_3=8$. The left one is homeomorphic to the product of a Cantor set with a circle, pinched at one fiber (to achieve this, the map $T_1$ pinches two fibers to a common point); the topology of the other two attractor is very intricate, in particular their fundamental groups seem not to be finitely generated -- we expect these attractors to be homotopic to complement of Cantor sets in the plane, which would make them homotopic one to the other; they could even be homeomorphic.}
\label{fig:attractors}
\end{center}
\end{figure}

\subsection{Physical versus SRB}

We close this gallery of examples by stressing the difference between physical measures and SRB measures, where we use the distinction advocated by Young \cite{young2002srb} (see section \ref{sec:def-physicality} and in particular Definition \ref{defi:SRB}).

Theorem \ref{theo:mainLift} combines with a result of Campbell and Quas \cite{campbell2001generic} (where they say ``SRB'' for ``physical'') to yield the following.
\begin{corollary}\label{coro:mainSRB}
There is a $\C^1$ diffeomorphism onto its image $T:U\to U$, where $U\subset \mathbb{R}^3$ is open and  bounded, having a compact uniformly hyperbolic attractor $\Lambda = \bigcap_{k\ge 0} T^k(U)$ with unstable dimension $1$, supporting exactly one physical measure $\mu_T$, but supporting no SRB measure.
\end{corollary}
The proof, detailed in Section \ref{sec:proofs}, can be summed up easily as follows. We construct $T$ as a uniformly hyperbolic skew product $(S,R)$ on $\mathbb{T}\times D^2$; \cite{campbell2001generic} shows that taking $S$ a generic $\C^1$ expanding map, it has a unique physical measure $\check\mu$, but no absolutely continuous measure. The lift $\mu$ of $\check\mu$ is a physical measure of $T$ with full basin of attraction, but an SRB measure would project to an absolutely continuous measure of $S$ and thus does not exist. In other word, \cite{campbell2001generic} already provides many examples of the kind above, but somewhat degenerate as the stable dimension vanishes; the present work only serves to add some stable dimensions.

\begin{remark}\label{rema:hyperbolic}
When $T:M\to M$ is a $C^2$ diffeomorphism with an hyperbolic attractor $\Lambda = \cap_{n_in\mathbb{N}} T^n(U)$ where $U$ is an open set with $T(\bar U)\subset U$, and when in addition there exist a compact $\C^1$ submanifold $Y\in U$ transversal to the stable foliation and intersecting each stable leaf at exactly one point, one can identify $Y$ with the space of leaves and gets a factor $S:Y\to Y$ of $T$ with a nice section $Y\hookrightarrow U$ (for example, if the unstable dimension is $1$ and stable leaves on $U$ are relatively compact, such a $Y$ is easily constructed). One can easily endow $Y$ with a (non-necessarily Riemannian) metric that makes $S$ expanding. On the one hand, it follows from Theorem 5.8 in \cite{kloeckner2017optimal} that $S$ satisfies $\UE(\hol{*};\mathrm{CLT}, \hol{*})$; on the other hand the geometric potential $\varphi_T=-\log J^uT$ (where $J^u$ is the determinant of the restriction of $DT$ to the stable distribution) is $\C^1$, in particular H\"older, and one can thus use the present results to construct a H\"older potential $\hat{\varphi}_T$ cohomologuous to $\varphi_T$ and constant on stable leaves. This potential descends on $Y$ into a H\"older potential, for which $S$ has a unique equilibrium state. It follows that $T$ has a unique equilibrium state $\mu_T$ for the geometric potential, and that $\mu_T$ has the expected statistical properties (CLT, exponential decay of correlations). By the work of Ledrappier \cite{ledrappier1984proprietes}, we know that such an equilibrium state is a SRB measure. Moreover, it is a physical measure thanks to the absolute continuity of holonomy of the stable foliation (see e.g. \cite{abdenur2009flavors}). Without relying on Markov partitions, we recover in this case the classical result that there is a measure $\mu_T$ that is the unique physical measure, the unique SRB measure and the unique equilibrium state of the geometric potential; and that $\mu_T$ has good statistical properties.
\end{remark}

\subsection{Beyond Lipschitz maps with uniformly shrinking fibers}

The present work can be developed in several directions; for example one could apply similar ideas for flows. Section \ref{sec:lift} only uses a averaged shrinking, and could thus be applied to the examples introduced by Di\'{\i}az, Horita, Rios and Sambarino \cite{diaz2009destroying} and further studied by Leplaideur, Oliveira and Rios \cite{leplaideur2011equilibrium} and Ramos and Siqueira \cite{ramos2017equilibrium}. These examples, which are at the frontier between hyperbolic and robustly non-hyperbolic dynamic, are indeed extensions of uniformly expanding maps on a Cantor subset of $\mathbb{R}$, with only some exceptional fibers not being contracted. The ideas of the other sections might be applicable to such examples as well.

As mentioned above, in some interesting cases the map $S$ is not continuous, see e.g. \cite{galatolo2018quantitative}, \cite{galatolo2018decay}. We expect most of the ideas used to prove Theorems \ref{theo:mainSG}-\ref{theo:maindecay} to be adaptable to such a setting, up to devising suitable functional spaces to work on (we do not claim that such an adjustment should always be straightforward); this should make it possible to consider more general invariant measures than the lift of the absolutely continuous $S$-invariant measure. In particular, using disintegration with respect to $\pi$ in its full generality should be useful. 

Note also that the ideas presented here can be used without assuming compact fibers, if one has \emph{contraction} properties instead of shrinking (e.g. $d(Tx,Tx') \le \lambda d(x,x')$ for some $\lambda\in(0,1)$ and all $x,x'$ in the same fiber, or milder contraction using decay functions as in \cite{kloeckner2017optimal}). One would need to assume some moment condition on the lifted measures $\mu$, in order to use a Wasserstein distance (and its modification $\wass^{\check\mu}$). In the examples I could think of, it would be possible to actually restrict to a stable compact subset of the space though, so that we do not pursue this direction.

\subsection{Organization of the article}

In Section \ref{sec:prelim}, we introduce a number of tools and definitions, many being very classical. Given the variety of properties considered in our main theorem, this section is rather long for a preliminary one. Each section after that starts by pointing to the subsections \ref{sec:prelim}.$\ast$ that are used, so that Section \ref{sec:prelim} can be mostly skipped and used as reference.

In Section \ref{sec:lift} we prove existence, uniqueness of the $T$-invariant lift of an $S$-invariant measure and study convergences to it under iteration of $T$ (this covers the first part of Theorem \ref{theo:mainLift}, and Theorem \ref{theo:mainSG}). Section \ref{sec:properties} ends the proof of Theorem \ref{theo:mainLift} by considering each preserved property. In this part we consider the more general case of fibers shrunk on average with respect to an $S$-invariant measure, while all the following sections assume that fibers are all (uniformly) shrinking.

In Section \ref{sec:equilibrium}, we consider equilibrium states and establish a correspondence between potentials and observables on $X$ and on $Y$, by adding coboundaries and using the projection. Theorem \ref{theo:mainstat} is in particular proved.

Section \ref{sec:decay} is devoted to the decay of correlations, and proves Theorem \ref{theo:maindecay}. To this end, we use another correspondence between observables on $X$ and on $Y$, using disintegration; we prove that disintegration preserve some regularity properties of observables (Theorem \ref{theo:regularity}).

Last, in Section \ref{sec:proofs} we explain how to deduce Corollaries \ref{coro:mainPM}-\ref{coro:mainSRB} from the main theorems and the literature.

\paragraph*{Acknowledgement.}
I warmly thank Stefano Galatolo for many interesting comments on a preliminary version of this work.

\section{Preliminaries}\label{sec:prelim}

This Section sets up notation and states a few results we shall use.

Let $X$ be a compact metric space and $T:X\to X$ be a map (all maps are assumed to be Borel-measurable), admitting a factor $S:Y\to Y$, i.e. $Y$ is a compact metric space, $S$ is a map, and there is a continuous onto map $\pi:X\to Y$ such that $S \pi = \pi T$ (we denote composition of maps either by juxtaposition of using the usual symbol $\circ$). The sets $\pi^{-1}(y)$ are called \emph{fibers}.
We denote by $d(\cdot,\cdot)$ both metrics on $X$ and on $Y$, the context preventing any ambiguity.
Note that to state the more general results, we do not ask $S,T$ to be continuous unless specified; on the contrary the continuity of $\pi$ is crucial in many arguments.

We denote by $\alg{B}_X$, $\alg{B}_Y$ the Borel $\sigma$-algebras of $X$ and $Y$,with respect to which all measurability conditions are considered unless otherwise specified. Let $\proba(X)$, $\proba(Y)$ be the sets of probability measures of $X$, $Y$ and $\proba_T(X)$ the set of $T$-invariant probability measures (similarly $\proba_S(Y)$ is the set of $S$-invariant probability measures). We denote either by $\int f\dd\mu$, $\int f(x) \dd\mu(x)$ or $\mu(f)$ the integral of a integrable or positive function $f$ with respect to a measure $\mu$.

In order to simplify a few arguments, we always assume (up to changing the metrics by a constant, thus not altering the statements of the Theorems in the introduction) that $\diam X,\diam Y\le 1$.

Constants denoted by $C$ are positive and can vary from line to line, and we write $a(n)\lesssim b(n)$ to express that for some $C>0$ and all $n\in\mathbb{N}$, $a(n) \le C b(n)$.

\subsection{Moduli of continuity}\label{sec:def-moduli}

By a \emph{modulus of continuity} we mean a continuous, increasing, concave function $\omega:[0,\infty) \to[0,\infty)$ mapping $0$ to $0$. We may only define a modulus near $0$, then the understanding is that it is extended to the half line; since we shall only be concerned with compact spaces, the specifics of the extension are irrelevant.

A function $f:X\to\mathbb{R}$ is said to have modulus of continuity $\omega$ when
\[ \lvert f(x)- f(x') \rvert \le \omega(d(x,x')) \qquad \forall x,x'\in X.\]
Every continuous function on a compact metric space is uniformly continuous, hence has a modulus of continuity: concavity of the modulus can be ensured by taking the convex hull of $\{(\eta,\varepsilon)\in [0,\infty)^2 \mid \exists x,x'\in X, d(x,x')\le \eta, \lvert f(x)-f(x')\rvert \ge \varepsilon \} $.

A function is said to be $\omega$-continuous if there is a constant $C>0$ such that it has $C\omega$ as a modulus of continuity; the infimum of all such $C$ is denoted by $\Hol_\omega(f)$, and the set of $\omega$-continuous functions $X\to\mathbb{R}$ is a Banach space (``generalised H\"older space'') when endowed with the norm
\[\lVert f\rVert_{\omega} := \lVert f\rVert_\infty + \Hol_\omega(f)\]
(this claim follows from the corresponding classical claim for the Lipschitz modulus and from the observation that $\omega(d(\cdot,\cdot))$ defines a metric). An observation that will be used without warning is that whenever $f$ has zero average with respect to an arbitrary probability measure, then it takes both non-positive and non-negative values; then $\lVert f\rVert_\infty\le \omega(\diam X)\Hol_\omega(f)$ and thus $\lVert f\rVert_\omega \lesssim \Hol_\omega(f)$.

The most classical moduli of continuity are the H\"older ones, defined for $\alpha\in(0,1]$ by $\hol{\alpha}(r) = r^\alpha$ (so that $\hol{\alpha}$-continuous means $\alpha$-H\"older). We shall have use for a family of more lenient moduli.

\begin{definition}
For each $\alpha\in(0,\infty)$ we denote by $\holl{\alpha}$ the modulus of continuity such that on $(0,1]$
\[\holl{\alpha}(r) = \frac{1}{\big(\log\frac{r_\alpha}{r}\big)^\alpha}\]
where $r_\alpha>1$ is chosen large enough to ensure monotony and concavity on $(0,1]$, and $\holl{\alpha}$ is constant for $r\ge 1$.

A $\holl{\alpha}$-continuous function is also said to be
$\alpha \log$-H\"older; a function is said to be $\log$-H\"older if it is $\alpha\log$-H\"older for some $\alpha>0$.
\end{definition}

To simplify notation, we write $\Hol_\alpha$ instead of $\Hol_{\hol{\alpha}}$ and $\Hol_{\alpha \log}$ instead of $\Hol_{\holl{\alpha}}$.

\begin{example}
When $X=\{0,1\}^{\mathbb{N}}$ with the metric $d(x,x') = 2^{-i(x,x')}$, where $i(x,x')$ is the first index where $x_i\neq x'_i$, a function $f$ is H\"older-continuous when the maximal influence of the $i$-th component decays exponentially fast, while $f$ is $\alpha\log$-H\"older when the maximal influence of the $i$-th component decays like $i^{-\alpha}$.
\end{example}

Let us show that the modulus $\holl{\alpha}$ being very concave, it is only mildly affected by pre-composition by a high-order iterate of a Lipschitz map.

\begin{proposition}\label{prop:alphalog}
For all $\alpha>0$ and all $L\ge 1$, there exists $D>0$ such that for all $n\in\mathbb{N}$ and all $r\in[0,1]$:
\[\holl{\alpha}(L^n r)\le D  n^\alpha \holl{\alpha}(r).\]
In particular, if $T:X\to X$ and $\sigma:Y\to X$ are Lipschitz and $f \in\Hol_{\alpha\log}(X)$, then 
$\Hol_{\alpha\log}(fT^n\sigma) \lesssim n^\alpha \Hol_{\alpha\log}(f)$.
\end{proposition}

\begin{proof}
When $L^n r\le 1$, we have
\[\frac{\holl{\alpha}(L^n r)}{\holl{\alpha}(r)} = \Big(1-\frac{\log L^n}{\log\frac{r_\alpha}{r}}\Big)^{-\alpha}\]
which is bounded independently of $n$ since $L^n \le \frac1r < \frac{r_\alpha}{r}$; while when $L^n r\ge 1$, we have
\[\frac{\holl{\alpha}(L^n r)}{\holl{\alpha}(r)} = \frac{\holl{\alpha}(1)}{\holl{\alpha}(r)} \lesssim \big(\log\frac{r_\alpha}{r}\big)^\alpha \lesssim (n \log L)^\alpha. \]
\end{proof}

\subsection{Sections, disintegration}\label{sec:def-disintegration}

The map $\pi$ can be used to push measures forward: given $\mu\in \proba(X)$, $\pi_*\mu : A \mapsto \mu(\pi^{-1} A)$ is a probability measure on $Y$. Moreover, a $T$-invariant measure is pushed to an $S$-invariant measure: for all $f:Y\to \mathbb{R}$,
\[ \int f\circ S \dd(\pi_*\mu) = \int f\circ S\pi \dd \mu = \int f\circ \pi T \dd\mu = \int f\circ \pi \dd\mu = \int f \dd(\pi_*\mu).\]
Our first goal, in Section \ref{sec:lift}, will be to lift an invariant measure $\check\mu$ of $S$ into an invariant measure $\mu$ of $T$, where ``lifting'' entails that $\pi_*\mu=\check\mu$.

Recall a notion borrowed from the theory of fiber bundles.
\begin{definition}
A \emph{section}  of $\pi$ is a measurable map $\sigma : Y\to X$ such that $\pi\sigma = \Id_Y$.
\end{definition}
In other words, $\sigma(y)\in\pi^{-1}(y)$ for all $y\in Y$, i.e. $\sigma$ picks a point in the fiber of its argument. The map $\sigma\pi: X\to X$ then sends each point to the point in its own fiber picked by $\sigma$. Note that there is no assumption relating the section with the dynamics. 
Asking $\sigma$ to be measurable is very mild, and in many cases we will ask it to be continuous, or even Lipschitz.
\begin{proposition}[Measurable Selection Theorem \cite{kuratowski1965general}]
There exist a section $\sigma:Y\to X$. As a consequence, $\pi_*$ is onto $\proba(Y)$.
\end{proposition}
(That $\pi_*$ is onto follows from the observation that for all $\nu\in\proba(Y)$, $\pi_*(\sigma_*\nu) = \nu$.)

We shall use in a central way the disintegration along a map. We state here the Disintegration Theorem for $\pi$, but it only needs measurability and can be used with other maps, such as $S$.
\begin{proposition}[Disintegration Theorem \cite{rohlin1952fundamental}, \cite{simmons2012conditional}]\label{prop:disintegration}
Let $\mu\in\proba(X)$ and $\check\mu=\pi_*\mu \in\proba(Y)$ be (non necessarily invariant) probability measures. There exist a family $(\xi_y)_{y\in Y}$ of probability measures on $X$ such that $y\mapsto \xi_y$ is Borel-measurable, $\xi_y$ is concentrated on $\pi^{-1}(y)$ for all $y\in Y$, and
\[\int \xi_y(f) \dd\check\mu = \int f \dd\mu \qquad \forall f\in L^1(\mu).\]
Moreover $(\xi_y)_{y\in Y}$ is uniquely defined by these properties up to a $\check\mu$-negligible set.
\end{proposition}
For example, if $\mu=\sigma_*\check\mu$ for some section $\sigma$, then $\xi_y = \delta_{\sigma(y)}$ for $\check\mu$-almost all $y\in Y$.

Given a function $f:X\to\mathbb{R}$ in $L^1(\mu)$, we can define a function in $L^1(\check\mu)$ by $\xi(f):y\mapsto \xi_y(f)$ (and then $\check\mu(\xi(f))=\mu(f)$ and $\xi(f)(y)$ only depends on the values of $f$ on $\pi^{-1}(y)$). In Section \ref{sec:decay}, we shall study how much regularity $\xi(f)$ retains from the regularity of $f$; but we have to keep in mind that even when $f$ is continuous, $\xi(f)$ is unambiguously defined only modulo a $\check\mu$-negligible set.
\begin{definition}
We say that a measurable function $u:Y\to \mathbb{R}$ \emph{has a continuous version} if there exist a continuous $\bar u:Y\to \mathbb{R}$ which is equal to $u$ at $\check\mu$-almost every point. If $\supp\check\mu=Y$, then $\bar u$ is unique.

We say that $\xi$ \emph{preserves continuity} if for all continuous $f:X\to\mathbb {R}$, $\xi(f): y\mapsto \xi_y(f)$ has a continuous version. 

If $\omega,\check\omega$ are two moduli of continuity, we say that $\xi$ is \emph{$(\omega,\check\omega)$-bounded} if for all $\omega$-continuous $f$, $\xi(f)$ is $\check\omega$-continuous and moreover the linear map $f\mapsto \xi(f)$ is a continuous operator $\Hol_\omega(X) \to \Hol_{\check\omega}(Y)$.
If $\check\omega=\omega$, then we simply say that $\xi$ is $\omega$-bounded.
\end{definition}

\subsection{Physicality, observability, SRB}\label{sec:def-physicality}

Assume now that $X$ and $Y$ are equipped with measures $\lambda_X$ and $\lambda_Y$ (which \emph{a priori} need not have any particular relation with $T$, $S$ but will serve as reference measure), e.g. $X,Y$ are manifolds equipped with volume forms, or are domains of $\mathbb{R}^m, \mathbb{R}^{\check m}$ equipped with the Lebesgue measure.

\begin{definition}
The \emph{Basin} of a $T$-invariant measure $\mu$ is defined as
\[ \basin(\mu) = \Big\{ x\in X \,\Big|\, \frac1n \sum_{k=0}^{n-1} \delta_{T^kx} \to \mu \Big\} \]
(where $\to$ denotes weak-$*$ convergence and $\delta_x$ is the Dirac mass at $x$).

The invariant measure $\mu$ is said to be \emph{physical} if its basin
 has a positive volume: 
 \[\lambda_X(\basin(\mu))>0.\]
\end{definition}
Often, physical measure are said to be the ones that can be seen in practice, given they drive the behavior of a positive proportion of the points. However, note that in some cases Guihéneuf \cite{guiheneuf2015dynamical} has shown that non-physical measures could be actually observed.

Physical measures do not always exist, and a more general class of measure was proposed in \cite{catsigeras2011srb}.
\begin{definition}
Given $x\in X$, denote by $\po(x)\subset \proba(X)$ the set of cluster points of the sequence $\big( \sum_{k=0}^{n-1} \delta_{T^kx} \big)_n$. Observe that $\po(x)\subset\proba_T(X)$. 

Given $\mu\in \proba(X)$ and $\varepsilon>0$, the $\varepsilon$-basin of $\mu$ is
\[\basin_\varepsilon(\mu) = \{x\in X \mid \wass(\po(x),\mu)<\varepsilon \}. \]

An invariant measure $\mu\in \proba_T(X)$ is said to be \emph{observable} when for all $\varepsilon>0$, its $\varepsilon$-Basin has positive volume.
(Note that while choosing $\wass$ as metric has an influence on the $\varepsilon$-basins, any metric inducing the weak-$*$ topology yields the same notion of observability.)
\end{definition}

It is important to make the distinction between the above notion of \emph{physical} measures, the related but distinct notion of \emph{SRB} measure (beware some authors use ``SRB'' instead of ``physical''). To introduce SRB measure we need to recall some preliminary definitions.

Let $U$ be an open bounded set of a manifold $M$; a diffeomorphism onto its image $T:U\to U$ has a stable subset $\Lambda := \bigcap_n T^n(U)$, called its \emph{attractor}, on which $T$ induces a homeomorphism. Assuming that $\overline{T(U)} \subset U$ (equivalently, that the closure \emph{in $U$} of $T(U)$ is compact), the attractor is compact: we have
$\overline{T^{n+1}(U)} = T^n(\overline{T(U)}) \subset T^{n}(U) \subset \overline{T^{n}(U)}$, so that $\Lambda = \bigcap_n \overline{T^{n}(U)}$ is a decreasing intersection of compact sets.

One says that $\Lambda$ is a \emph{strongly partially hyperbolic} attractor whenever there are continuous sub-bundles $E^u,E^s$ of $T_\Lambda M$ (the restriction of the tangent bundle $TM$ to $\Lambda$) of respective dimension $d_u$, $d_s$, and there are numbers $C$ and $\lambda_+ > \lambda_- > 0$ such that for all $x\in \Lambda$ and all $n\in\mathbb{N}$:
\[ \lVert DT^n_x(u) \rVert \le C\lambda_-^n \lVert u\rVert \quad\forall u\in E^s_x \quad\text{and}\quad  \lVert DT^n_x(u) \rVert \ge C^{-1}\lambda_+^n \lVert u\rVert \quad\forall u\in E^u_x  \]
If moreover $\lambda_-<1<\lambda_+$, one says that $\Lambda$ is a \emph{uniformly hyperbolic} attractor.

Assuming $\Lambda$ is a uniformly hyperbolic attractor, we get an invariant ``stable'' foliation $W^s$ of $U$, whose leaf $W^s_x$ through $x\in \Lambda$ is the set of points whose orbit converge to the orbit of $x$; and an invariant ``unstable'' lamination $W^u$ of $\Lambda$, whose leaf $W^u_x$ through $x$ is the set of points whose backward orbit converges to the backward orbit of $x$. The leaves of $W^s$ and $W^u$ are $\C^1$, and they are continuous but not necessarily transversely $\C^1$. Moreover $T_xW^s = E^s_x$ and $T_xW^u = E^u_x$.

Locally, we can then write the attractor $\Lambda$ as a product (one factor corresponding to the unstable direction, the other to the stable direction). Given an invariant measure $\mu\in\proba_T(U)$ (which must be supported on $\Lambda)$, we can disintegrate the restriction of $\mu$ to a small open set of $\Lambda$ with respect to the (local) projection $\pi^s$ on the stable direction, obtaining one one hand a family of measures $(\mu_L)_L$ supported on each local unstable leaf $L$, and on the other hand a projected measure $\nu=\pi^s_*\mu$.
\begin{definition}\label{defi:SRB}
We say that $\mu$ is an \emph{SRB} measure when in this local disintegration, $\mu_L$ is absolutely continuous with respect to the Riemannian volume induced on $L$ for $\nu$-almost all $L$.
\end{definition}

\subsection{Wasserstein metric and its vertical version}\label{sec:def-Wasserstein}

We will make use of the Wasserstein metric to metrize the weak-$*$ topology on $\proba(X)$. It is defined for $\mu_0,\mu_1\in\proba(X)$ by
\[\wass(\mu_0,\mu_1) = \inf_{\gamma\in\Gamma(\mu_0,\mu_1)} \int d(x_0,x_1) \dd\gamma(x_0,x_1) = \sup_{f\ 1-\text{Lipschitz}} \big\lvert \mu_0(f) -\mu_1(f) \big\rvert\]
where $\Gamma(\mu_0,\mu_1)$ is the set of \emph{couplings}, or \emph{transport plans} between $\mu_0$ and $\mu_1$, i.e. the set of $\gamma\in \proba(X\times X)$ such that $\gamma(A\times X) = \mu_0(A)$ and $\gamma(X\times A) = \mu_1(A)$ for all Borel $A\subset X$. The equality between the two definitions (by transport plans or by duality with Lipschitz functions) is not trivial, and is called Kantorovich duality. The infimum is reached, any transport plan realizing it is said to be optimal, and the set of optimal plans is compact in the weak-$*$ topology (see e.g. \cite{villani2009oldandnew}).

To prove Theorem \ref{theo:lift} below we introduce a variation of the Wasserstein metric where mass is only allowed to move along fibers. This constraint implies that we need to consider pairs of measure with the same projection. Similar ideas have been developed in \cite{galatolo2010lorenz}, \cite{araujo2014decay} and \cite{galatolo2015spectral}, but in somewhat restricted settings, without taking full advantage of the dual formulations of the Wasserstein metric and of the disintegration theorem.

For each $\check\mu\in \proba(Y)$, by continuity of $\pi$ the fiber $\pi_*^{-1}(\check\mu)$ is a closed subset of $\proba(X)$ in the weak-$*$ topology, thus is compact. 
Set $\Delta_\pi = \{(x_0,x_1)\in X\times X \mid \pi(x_0) = \pi(x_1) \}$. 
Given any $\mu_0,\mu_1 \in\pi_*^{-1}(\check\mu)$, we denote by $\Gamma_\pi(\mu_0,\mu_1)$ the set of $\gamma\in \Gamma(\mu_0,\mu_1)$ which are concentrated on $\Delta_\pi$. We define:
\[\wass^{\check\mu} (\mu_0,\mu_1) = \inf_{\gamma\in\Gamma_\pi(\mu_0,\mu_1)} \int d(x_0,x_1) \dd\gamma(x_0,x_1).\]
We will see in a minute that this is a finite number, but it is already clear that $\wass\le \wass^{\check\mu}$.

\begin{lemma}\label{lemm:vertical}
As soon as $\pi_*\mu_0=\pi_*\mu_1$, the set $\Gamma_\pi(\mu_0,\mu_1)$ is non-empty, and as a consequence $\wass^{\check\mu}(\mu_0,\mu_1)<\infty$. More precisely, if $(\xi_y)_{y\in Y}$ and $(\zeta_y)_{y\in Y}$ are the disintegrations of $\mu_0$ and $\mu_1$ with respect to $\pi$, then
\[\wass^{\check\mu}(\mu_0,\mu_1) = \int \wass(\xi_y,\zeta_y) \dd\check\mu(y).\]
\end{lemma}

\begin{proof}
Choose measurably $\eta_y \in \Gamma(\xi_y,\zeta_y)$ for each $y\in Y$ (e.g. $\eta_y = \xi_y\otimes\zeta_y$), and
let $\gamma = \int \eta_y \dd\check\mu \in \proba(X\times X)$, i.e. for all Borel $A,B\in X$, $\gamma(A\times B) = \int  \eta_y(A\times B) \dd\check\mu$. Then $\gamma(A\times X) = \int \xi_y(A) \dd\check\mu = \mu_0(A)$ and $\gamma(X\times A) = \int \zeta_y(A) \dd\check\mu = \mu_1(A)$ so that $\gamma\in\Gamma(\mu_0,\mu_1)$. Moreover since $\eta_y$ projects to two measures supported on $\pi^{-1}(y)$, it is supported on $\pi^{-1}(y)\times \pi^{-1}(y)\subset \Delta_\pi$. It follows that  $\gamma$ is concentrated on $\Delta_\pi$ and $\gamma\in \Gamma_\pi(\mu_0,\mu_1)$. Any $\gamma\in \Gamma_\pi(\mu_0,\mu_1)$ is of the form $\int \eta_y \dd\check\mu$, $(\eta_y)_{y\in Y}$ being the disintegration of $\gamma$ with respect to the map induced by $\pi$ from $\Delta_\pi$ to $Y$. Since
$\int d(x,x')\dd\gamma(x,x') = \iint d(x,x') \dd\eta_y(x,x') \dd\check\mu(y) \ge \int \wass(\xi_y,\zeta_y) \dd\check\mu(y)$,
taking an infimum we get $\wass^{\check\mu}(\mu_0,\mu_1)\ge \int \wass(\xi_y,\zeta_y)\dd\check\mu$.

For each $y$, the set of optimal transport plans from $\xi_y$ to $\zeta_y$ is compact, thus by the measurable selection theorem there is a measurable family $(\eta_y)_{y\in Y}$ such that for $\check\mu$-almost all $y\in Y$, $\int d(x,x') \dd\eta_y(x,x') = \wass(\xi_y,\zeta_y)$. It follows $\wass^{\check\mu}(\mu_0,\mu_1)\le \int \wass(\xi_y,\zeta_y)\dd\check\mu$.
\end{proof}


\begin{proposition}\label{prop:complete}
For all $\check\mu\in\proba(Y)$, $\wass^{\check\mu}$ is a complete metric on the set $\pi_*^{-1}(\check\mu)$.
\end{proposition}

\begin{proof}
The expression $\int \wass(\xi_y,\zeta_y) \dd\check\mu(y)$ is the $L^1(\check\mu)$ metric for functions from $Y$ to the compact metric space $(\proba(X),\wass)$. By Lemma \ref{lemm:vertical}, $\wass^{\check\mu}$ is the restriction of this metric to the closed subset $\pi_*^{-1}(\check\mu)$. The claim thus follows from the Riesz-Fischer theorem for functions with values in a complete metric space.
\end{proof}

%
%

\subsection{Transfer operators, spectral gap and correlations}\label{sec:def-transfer}

Transfer operator are multifaceted objects that are both tools and a objects of study. We will use a definition that needs to introduce a generalization of the notion of invariant measure, and then we shall describe other equivalent definitions. 

\subsubsection{Quick introduction to transfer operators}

We consider $S$ acting on $Y$ since it is the level at which transfer operator will be most relevant.
\begin{definition}
A measure $\lambda_Y\in\proba(Y)$ is said to be a \emph{conformal} measure of $S$ when $S_*\lambda_Y$ is absolutely continuous with respect to $\lambda_Y$.

Given a conformal measure $\lambda_Y$ of $S$, one defines the \emph{transfer operator}
\[\chop{L}=\chop{L}_{\lambda_Y}:L^1(\lambda_Y) \to L^1(\lambda_Y)\]
of $S$ with respect to $\lambda_Y$ by $S_*(f\dd\lambda_Y) = \chop{L}(f) \dd\lambda_Y$.
\end{definition}
For example, the Lebesgue measure on the circle is a conformal measure for all local $\C^1$ diffeomorphisms (but not for a map that is constant on an interval). The transfer operator simply translates the action of $S_*$ on the set of absolutely continuous measures (with respect to $\lambda_Y$) to the space of densities. In particular, finding an absolutely continuous invariant measure is equivalent to finding a non-negative, non-zero eigenfunction of $\chop{L}$ (the eigenvalue is then necessarily $1$, since $\int \chop{L}(f)\dd\lambda_Y = \int f \dd\lambda_Y$).

Another classical way to say the same thing is to define $\chop{L}$ as the dual operator of the Koopman operator $f \mapsto f\circ S :L^\infty(\lambda_Y) \to L^\infty(\lambda_Y)$, i.e. to characterize it by the property
\begin{equation}
\int f \cdot  \chop{L}(g) \dd\lambda_Y = \int f\circ S \cdot g \dd\lambda_Y \qquad \forall f\in L^\infty(\lambda_Y),\ \forall g\in L^1(\lambda_Y).
\label{eq:transfer}
\end{equation}

Invariant measures $\check\mu$ are characterized by the property that their transfer operator have the property $\chop{L}_{\check\mu}\one = \one$ where $\one$ is the constant function with value one. 

\subsubsection{Decay of correlations and spectral gap}

Transfer operators are precious tools to study the decay of correlations.
\begin{definition}\label{defi:decay}
Given $\mu\in\proba_T(X)$ and functions $u,v : X\to\mathbb{R}$ (called observables), \emph{correlations} are defined (whenever it makes sense) as
\[\corr{n}{\mu}(u,v) = \Big\lvert \int u\circ T^n \cdot v \dd\mu - \int u\dd\mu \int v\dd\mu \Big\rvert\]
(and of course $\corr{n}{\check\mu}(f,g)$ with $f,g:Y\to\mathbb{R}$ implicitly involves the map $S$).

We say that $\mu\in\proba_T(X)$ \emph{has decay of correlations $(b_n)_{n\in\mathbb{N}}$ for $\omega$-continuous observables} whenever for all $f\in \Hol_{\omega}(X)$, all $g\in L^1(\mu)$ and all $n\in\mathbb{N}$,
\[\corr{n}{\mu}(f,g) \le b_n \Hol_\omega(f) \lVert g\rVert_{L^1(\mu)}. \]
\end{definition}

The link between the transfer operator and decay of correlation is quite direct: assuming $g\in L^1(\check\mu)$, $f\in L^\infty(\check\mu)$ and adding a constant to $f$ to ensure $\check\mu(f)=0$, we obtain
\[
\corr{n}{\check\mu}(g,f) = \Big\lvert \int g\circ S^n \cdot f\dd\check\mu \Big\rvert
  = \Big\lvert \int g \cdot \chop{L}_{\check\mu}^n(f) \dd\check\mu \Big\rvert 
  \le \lVert g \rVert_{L^1(\check\mu)} \lVert \chop{L}_{\check\mu}^n(f) \rVert_{L^\infty(\check\mu)}
\]
(other pairs of functional spaces can be considered, such as $L^2$ and $L^2$, or inverting the roles of $L^1$ and $L^\infty$). One thus only has to prove decay of $\chop{L}_{\check\mu}^n$ for zero-average observables in some functional space to obtain a corresponding decay of correlations.
A particularly nice case, both to find an Acip and to prove exponential decay of correlations for it, is when the transfer operator has a spectral gap.
\begin{definition}\label{defi:SG}
Given a Banach space $\Banach$ of functions $Y\to\mathbb{R}$ whose norm $\lVert\cdot\rVert$ is not less than $\lVert\cdot\rVert_{\infty}$ (one could generalize to $\lVert\cdot\rVert_{L^1(\check\mu)}$), one says that $\chop{L}$ has a \emph{spectral gap} on $\Banach$ whenever
\begin{itemize}
\item  $\chop{L}$ preserves $\Banach$ and acts on it as a bounded operator,
\item  there is a positive function $h\in \Banach$ such that $\chop{L} h = h$, which without lack of generality can be assumed to satisfy $\lambda_Y(h)=1$,
\item there are numbers $C\ge 1, \delta\in(0,1)$  such that for all $f\in\Banach$ with $\lambda_Y(f) = 0$, $\lVert \chop{L}^n f\rVert \le C (1-\delta)^n \lVert f\rVert$.
\end{itemize}
\end{definition}
Then $h\dd\lambda_Y$ is an $S$-invariant probability measure absolutely continuous with respect to $\lambda_Y$, satisfying exponential decay of correlations for observables in $\Banach$ (note that $\chop{L}$ and $\chop{L}_{h\dd\lambda_Y}$ are conjugated one to another).

We shall need some transfer operator to preserve some functional spaces in the following sense.
\begin{definition}
An operator $\op{P}: L^1(\check\mu) \to L^1(\check\mu)$ is said to \emph{preserve} $\Hol_\omega(Y)$ if $\op{P}(\Hol_\omega(Y))\subset \Hol_\omega(Y)$ and if it moreover induces a bounded operator on $\Hol_\omega(Y)$. The operator $\op{P}$ is said to be \emph{iteratively bounded} with respect to $\omega$ if it preserves $\Hol_\omega(Y)$ and if moreover there exist $D\ge 0$ such that for all $n\in\mathbb{N}$ and all $u\in\Hol_\omega(Y)$, $\Hol_\omega(\op{P}^n u) \le D \Hol_\omega(u)$. 
\end{definition}
For example, if $\op{P}$ has a spectral gap on $\Hol_\omega(Y)$, then it is iteratively bounded with respect to $\omega$ (but the latter assumption is much milder than having a spectral gap).

\subsubsection{Disintegrations and transfer operators}

To close this subsection, we shall consider a slightly different point of view on transfer operators, that seems novel in this generality (although it is folklore in the case of Lebesgue measure) and will enable us to relate the transfer operators of $T$ and $S$. We restrict to the case of a $T$-invariant measure $\mu$ and its $S$-invariant projection $\check\mu=\pi_*\mu$. The transfer operators of $(T,\mu)$ and $(S,\check\mu)$ are denoted by $\op{L}$ and $\chop{L}$.

We denote by $(\xi_y)_{y\in Y}$ the disintegration of $\mu$ with respect to the map $\pi$, which we recall is characterized by two properties: $\mu = \int \xi_y \dd\check\mu(y)$ and $\supp \xi_y \subset \pi^{-1}(y)$ for all $y\in Y$.
The disintegration theorem can also be applied to $T$ and $\mu$, and yields an essentially unique measurable family of probability measures $(\eta_x)_{x\in X}$ characterized by  $\mu = \int \eta_x \dd\mu(x)$ and $\eta_x(T^{-1}(x))=1$ for all $x\in X$ (here $T^{-1}(x)$ need not be closed, and while $\eta_x$ is \emph{concentrated} on $T^{-1}(x)$ its \emph{support} could be larger).  

To understand the meaning of this disintegration, one can say that each measure $\eta_x$ collects the ``derivatives'' of $T$ with respect to the measure $\mu$ at the points of $T^{-1}(x)$. The clearest situation is when $T$ is at-most-countable-to-one, in which case $\eta_x$ is atomic and the masses of the atoms can be taken as definition of this derivative. When $T$ is one-to-one, then of course $\eta_x = \delta_{T^{-1}x}$.

We can use the disintegration to express the transfer operator.
\begin{proposition}
We have $\op{L} g(x) = \eta_x(g)$ for all $g\in L^1(\mu)$ and $\mu$-almost all $x\in X$.
\end{proposition}

\begin{proof}
The proposed formula defines a bounded operator $\tilde{\op{L}}(g)(x) = \eta_x(g)$ of $L^1(\mu)$ into itself, and to prove $\tilde{\op{L}}=\op{L}$ it suffices to check the defining property \eqref{eq:transfer}. Let $f\in L^\infty(\mu)$ and $g\in L^1(\mu)$; using that $x=Tx'$ for $\eta_x$-almost all $x'$, we get
\[\int f \cdot  \tilde{\op{L}}g \dd\mu = \int f(x) \int g(x') \dd\eta_x(x') \dd\mu(x) = \iint f(Tx') g(x') \dd\eta_x(x') \dd\mu(x) 
  = \int f\circ T\cdot g \dd\mu.
\]
\end{proof}

Very often, one works the other way around: the family $(\eta_x)_{x\in X}$ is given and used to define a transfer operator, which is in turned used to construct an invariant measure $\mu$ with the prescribed derivatives. Using the disintegration theorem makes transparent the fact that one can go both ways round in a consistent fashion.

Note that, using either definition of transfer operator we easily get the classical property
$\op{L}(f\circ T \cdot g) = f \op{L}(g)$ (where $g\in L^1(\mu)$, $f\in L^\infty(\mu)$). We have $T_*\eta_x = \delta_x$ since $\eta_x$ is a probability measure supported on $T^{-1}(x)$.

The same applies to $S$, and  we denote by $(\check\eta_y)_{y\in Y}$ the disintegration of $\check \mu$ with respect to $S$ and by $\chop{L}$ its transfer operator. The same relations than above hold, in particular $\chop{L}u(y) = \check\eta_y(u)$. Now, our goal is to relate the transfer operators (or equivalently, the disintegrations) of $T$ and $S$.

\begin{lemma}
For $\check\mu$-almost all $y\in Y$, all $u\in L^1(\check\mu)$ and all $n\in\mathbb{N}$  we have 
\[\pi_*(\smallint \eta_x \dd\xi_y(x)) = \check\eta_y \qquad  \chop{L}^n u(y) = \xi_y\big(\op{L}^n(u\circ \pi) \big).\]
\end{lemma}

\begin{proof}
To prove the first claim, it suffices to check the two defining properties of $(\check\eta_y)_{y\in Y}$. First, the measure $\int \eta_x \dd\xi_y(x)$ is concentrated on $T^{-1}(\pi^{-1}(y)) = \{x\in X \mid \pi T(x) =y\} = \pi^{-1}(S^{-1}(y))$ so that its push-forward by $\pi$ is concentrated on $S^{-1}(y)$. Second, for all $u\in L^1(\check\mu)$ we have
\begin{multline*}
\int \pi_*(\smallint \eta_x \dd\xi_y(x))(u)  \dd\check\mu(y) = \iint \eta_x(u\circ \pi) \dd \xi_y(x) \dd\check\mu(y) \\
  = \int \eta_x(u \circ \pi) \dd\mu(x)
  = \int u\circ\pi \dd\mu 
  = \int u \dd\check\mu.
\end{multline*}

We prove the second claim using the duality definition of $\chop{L}$:
\begin{multline*}
\int \xi_y\big(\op{L}^n(u\pi)\big) \cdot v(y) \dd\check\mu(y) 
  = \int \xi_y\big(\op{L}^n(u\pi) \cdot v\pi \big) \dd\check\mu(y) \\
  = \int \op{L}^n(u \pi) \cdot v\pi \dd\mu
  = \int u\pi \cdot v S^n \pi \dd\mu 
  = \int u\cdot v S^n \dd\check\mu 
  = \int \chop{L}^n(u) \cdot v \dd\check\mu.
\end{multline*}
\end{proof}

\subsection{Statistical properties}\label{sec:def-stat}

Let us define precisely the three statistical properties we shall focus on (as will be clear from the proofs, we could consider any statistical theorem insensitive to adding a bounded error term to $\sum_{k=1}^n f\circ T^k$.)

\begin{definition}\label{defi:stat}
Let $\LimTh{T}\in\{\mathrm{LIL},\mathrm{CLT},\mathrm{ASIP}\}$;
we shall say that an invariant measure $\mu\in\proba_T(X)$ \emph{satisfies $\LimTh{T}$ for all $\omega$-continuous observables} if for each $f\in \Hol_{\omega}(X)$ there is $\sigma_f\ge 0$ (meant as a standard deviation, not to be confused with a section) such that, whenever $\sigma_f>0$:
\begin{description}
\item[When $\LimTh{T}=\mathrm{LIL}$ (Law of Iterated Logarithm):]  for $\mu$-almost every $x\in X$
\[ \limsup_{n\to \infty} \frac{\sum_{k=1}^n f\circ T^k(x) -n \mu(f)}{\sqrt{2 n\log \log n}} = \sigma_f,\]
\item[When $\LimTh{T}=\mathrm{CLT}$ (Central Limit Theorem):] denoting by $G_{m,\sigma^2}$ the cumulative distribution function of the normal law of mean $m$ and variance $\sigma^2$, for all $r\in \mathbb{R}$
\[ \mu\Big\{x\in X : \frac1{\sqrt{n}} \sum_{k=1}^n f\circ T^k(x) \le r \Big\} \to G_{\mu(f),\sigma_f^2}(r), \]
\item[When $\LimTh{T}=\mathrm{ASIP}$ (Almost sure Invariance Principle):] for some $\lambda\in(0,\frac12]$, there exist a probabilistic space $\Omega$ and two real-valued processes defined on $\Omega$: 
\begin{itemize}
\item $(A_k)_{k\in\mathbb{N}}$ with the same law as $(f\circ T^k(Z))_{k\in\mathbb{N}}$ where $Z$ is a random variable with law $\mu$; 
\item $(B_k)_{k\in\mathbb{N}}$, a sequence of independent Gaussian random variables of mean $\mu(f)$ and variance $\sigma_f^2$
\end{itemize}
such that almost surely $\big\lvert \sum_{k=1}^n A_k -\sum_{k=1}^n B_k \big\rvert = o(n^\lambda)$.
\end{description}
We will usually keep the data $\sigma_f$, $\lambda$ implicit but they are part of the statistical theorem, and when we state that a $\UE(\omega'_p[\rho];\LimTh{T},\omega'_o)$ property for $S$ implies a $\UE(\omega_p[C\rho];\LimTh{T},\omega_o)$ property for $T$, we always implicitly mean that the equilibrium state $\mu_{\varphi}$ of a potential $\varphi$ and $\check\mu:=\pi_*(\mu_\varphi)$ (which will be an equilibrium state for a potential $\check\varphi$) satisfy $\LimTh{T}$ \emph{with the same parameters} under a correspondence $f\mapsto \check f$ (made explicit in Section \ref{sec:equilibrium}), i.e. $\sigma_f = \sigma_{\check f}$ in the respective statements for $\mu$ and $\check\mu$ (and, in the case of the ASIP, additionally $\lambda$ is the same in both statements).
\end{definition}

\section{Lifting invariant measures}\label{sec:lift}

This Section uses the material of preliminary subsections \ref{sec:def-moduli},  \ref{sec:def-Wasserstein}, and \ref{sec:def-transfer} for the proof of Theorem \ref{theo:mainSG}.

\subsection{Existence and uniqueness}\label{sec:existence-uniqueness}

It is proved in \cite{araujo2009singular} that  under a shrinking hypothesis  each $\check\mu\in\proba_S(Y)$ has a lift $\mu\in \proba_T(X)\cap(\pi_*)^{-1}(\check\mu)$. Uniqueness seems to be only known under some ergodicity hypotheses (see \cite{butterley2017disintegration}, Remark 2 (b)). Our first result gives uniqueness in general and a quantified convergence, and generalizes to the case of fibers shrunk \emph{on average} (Definition \ref{defi:shrinking-average}).
\begin{theorem}[Lifting Theorem]\label{theo:lift}
Let $\check\mu$ be an $S$-invariant probability measure, and assume that the fibers of the extension $T$ are shrunk on average with respect to $\check\mu$, with shrinking sequence $(\bar a_n)_n$. Then there is a unique $\mu\in \proba_T(X)$ such that $\pi_*\mu = \check\mu$. Moreover, for all $\nu\in \proba(X)$ (not necessarily invariant) such that $\pi_*\nu = \check\mu$ and all $n\in\mathbb{N}$ we have
$\wass(T_*^n \nu,\mu) \le \bar a_n$;
in particular, $T_*^n\nu \to \mu$ in the weak-$*$ topology.

If $T$ has $(a_n)_n$-shrinking fibers then the above holds for all $\check\mu\in\proba_S(Y)$ with $\bar a_n=a_n$.
\end{theorem}

\begin{proof}
For all $\mu_0,\mu_1 \in \pi_*^{-1}(\check\mu)$, and all $\gamma\in\Gamma_\pi(\mu_0,\mu_1)$, denoting by $(T^n,T^n)$ the map from $X\times X$ to itself sending $(x_0,x_1)$ to $(T^n x_0,T^nx_1)$ we have \[(T^n,T^n)_*\gamma \in \Gamma_\pi(T^n_*\mu_0, T^n_* \mu_1).\]
Since $\gamma$ is supported on $\Delta_\pi$, for $\gamma$-almost all $(x_0,x_1)$ we have $\pi(x_0)=\pi(x_1)$; using this, that the first marginal of $\gamma$ is $\mu_0$ and that $\pi_*\mu_0=\check\mu$, we have:
\begin{align}
\int d(x_0,x_1) \dd \big((T^n,T^n)_*\gamma\big)(x_0,x_1) &= \int d(T^n x_0,T^nx_1) \dd \gamma(x_0,x_1) \nonumber\\
  &\le \int \diam\big(T^n(\pi^{-1}(\pi(x_0)))\big) \dd\gamma(x_0,x_1) \nonumber\\
  &\le\int \diam\big(T^n(\pi^{-1}(y))\big) \dd\check\mu(y) \nonumber\\
\wass^{\check\mu}(T^n_*\mu_0,T^n_*\mu_1) &\le  \bar a_n. \label{eq:lift}
\end{align}

Applying this to any $\nu\in\pi_*^{-1}(\check\mu)$ and to $T_*^{m-n}\nu$ we get 
$\wass(T_*^n\nu,T_*^m\nu) \le \bar a_n$ for all $n<m\in\mathbb{N}$, i.e. $(T_*^n\nu)_n$ is a Cauchy sequence with respect to $\wass^{\check\mu}$. By Proposition \ref{prop:complete}, it has a limit $\mu\in \pi_*^{-1}(\check\mu)$ in the metric $\wass^{\check\mu}$, which is also a weak-$*$ limit since $\wass\le \wass^{\check\mu}$.

Since $T$ is uniformly continuous along fibers, we have for any $\gamma\in\Gamma_\pi(T_*^n\nu,\mu)$
\[\int d(x,x') \dd\big((T,T)_*\gamma\big)(x,x') \le \int \bar\omega_T(d(x,x')) \dd\gamma(x,x') \le \bar\omega_T\big(\int d(x,x') \dd\gamma(x,x')\big).\]
Taking an infimum we get $\wass(T_*^{n+1}\nu,T_*\mu) \le \bar\omega_T\big(\wass^{\check\mu}(T_*^{n}\nu,\mu) \big)$; the left-hand side converges to $\wass(\mu,T_*\mu)$ while the right-hand side goes to $0$, so that $\mu$ is $T$-invariant.

Reapplying \eqref{eq:lift} to $\nu$ and $\mu$, we get the desired convergence in the Wassertein metric.
\end{proof}

\begin{remark}\label{rema:APPV}
The existence part in Corollary 6.2 in \cite{araujo2009singular} might at first seem more general in the case of shrinking fibers, as no continuity of the map $T$ (denoted there by $F$) is explicitly assumed while we assume uniform continuity \emph{along the fibers} in Definition \ref{defi:shrinking}. However, full continuity is implicitly used in the proof of Corollary 6.2 there: to obtain that $\mu_F$ is invariant, Lemma 6.1 is applied to the observable $\psi\circ F$, implicitly assuming it to be continuous.

A related issue found at other places in the literature is to construct $\mu\in(\pi_*)^{-1}(\check\mu)$ as the limit of push-forward measures $T^n_*\nu$, and deducing invariance of $\mu$ by writing
\[ T_*\mu=  T_*\big(\lim_{n\to\infty} T_*^n\nu\big) = \lim_{n\to\infty} T_*^{n+1}\nu = \mu.\]
This is perfectly fine when $T$ is continuous, but without this assumption the second equality may fail. 

A previous version of this article presented a very similar mistake, claiming that cluster points of empirical measures $\frac1n\sum_{k=0}^{n-1}\delta_{T^k x}$ from a point $x$ where invariant; an illustration of how this can fail when $T$ is discontinuous is given in Remark \ref{r:counter-ex}.
\end{remark}

\begin{corollary}\label{coro:homeo}
If $T$ is an extension of $S$ with shrinking fibers, then the map $\pi_*:\proba(X)\to\proba(Y)$ induces a homeomorphism from $\proba_T(X)$ to $\proba_S(Y)$.
\end{corollary}

\begin{proof}
By Theorem \ref{theo:lift}, $\pi_*$ induces a bijection  $\proba_T(X) \to \proba_S(Y)$. Since $\pi_*$ is continuous and $\proba_T(X)$ is compact, this induced map is a homeomorphism.
\end{proof}

We shall denote by $\pi^* : \proba_S(Y) \to \proba_T(X)$ the inverse map of this homeomorphism (this notation is differential-geometric flavored: an index star denotes push-forward while an exponent star denotes pull-back).

\subsection{Stable leafs of invariant measures}

Imagine that one wishes to draw a random point $x\in X$ whose law is close to $\mu$. Using Theorem \ref{theo:lift}, one could draw a random point $y\in Y$ with law $\check\mu$, choose in any way (random or deterministic) an inverse image $x_0\in\pi^{-1}(y)$, and take $x=T^n(x_0)$ for some large $n$. However, one may not be able to draw $y$ with precisely the law $\check\mu$. One would hopefully be able to draw $y$ with a law very close to $\check\mu$, and still get that the law of $x$ is close to $\mu$. 

In other words, one asks for conditions on a probability measure $\nu\in\proba_T(X)$ ensuring that $T^n_*\nu$ converges to a given invariant measure $\mu$. This idea also connects with the construction of SRB measures by iteratively pushing forward the Lebesgue measure. 
Define the \emph{stable leaf} of an invariant measure $\mu$ by
\[ \stabL(\mu,T_*) = \big\{ \nu\in \proba(X) \,\big|\, T_*^n\nu \to \mu \big\} \]
where the convergence is in the weak-$*$ topology. Since we are concerned here with the relations between $T$ and $S$, the question is to relate $\stabL(\mu,T_*)$ with $\stabL(\check\mu, S_*)$.

\begin{lemma}\label{lemm:stabL}
Let $\check\mu\in\proba_S(Y)$ and assume that
\begin{itemize}
\item the fibers are $(\bar a_n)_n$-shrunk on average with respect to $\check\mu$,
\item $T$ is continuous, and let $\omega_{T^k}$ be a modulus of continuity of $T^k$, for each $k\in\mathbb{N}$,
\item $\pi$ induces a continuous fibration with modulus $\bar\omega_\pi$.
\end{itemize}
Then for all $k,\ell\in \mathbb{N}$ and all $\nu\in\proba(X)$ we have
\[ \wass(T_*^{k+\ell} \nu,\mu) \le \omega_{T^k}\circ\bar \omega_\pi\big(\wass(S_*^\ell\pi_*\nu,\check\mu)\big) + \bar a_k. \]
\end{lemma}

\begin{proof}
We first prove that given any $\mu_0\in\proba(X)$ and any $\check\mu_1\in\proba(Y)$, there exist $\mu_1\in \pi_*^{-1}(\check\mu_1)$ such that
\[\wass(\mu_0,\mu_1) \le \bar\omega_\pi\big(\wass(\pi_*\mu_0,\check\mu_1)\big).\]
Let $\check\gamma$ be an optimal transport plan from $\pi_*\mu_0$ to $\check\mu_1$ and $(\xi_y)_{y\in Y}$ be the disintegration of $\mu_0$ with respect to $\pi$. Recall that $\tau_y^{y'}$ is a measurable map from $\pi{-1}(y)$ to $\pi^{-1}(y')$ such that $d(x,\tau_y^{y'}(x))\le \bar\omega_\pi(d(y,y')$. We define a measure on $X\times X$ by
$\gamma = \int (\Id,\tau_y^{y'})_*\xi_y \dd\check\gamma(y,y')$, i.e. for $f:X\times X\to\mathbb{R}$:
\[\gamma(f) = \iint f(x,\tau_y^{y'}(x)) \dd\xi_y(x) \dd\check\gamma(y,y').\]
The first marginal of $\gamma$ is $\mu_0$, since when $f$ only depends on its first argument
\[\gamma(f) = \iint f(x) \dd\xi_y(x) \dd\check\gamma(y,y') = \int f(x) \dd\xi_y(x) \dd\pi_*\mu_0(y) = \mu_0(f).\]
Let $\mu_1$ be the second marginal of $\gamma$; then $\pi_*\mu_1=\check\mu_1$ since when $f(x,x') = g(\pi(x'))$
\[\gamma(f) = \iint g(y') \dd\xi_y(x) \dd\check\gamma(y,y') = \int g(y') \dd\check\gamma(y,y')=\check\mu_1(g).\]
We get
\begin{multline*}
\wass(\mu_0,\mu_1) \le \int d(x,x') \dd\gamma(x,x')
  \le \int d(x,\tau_y^{y'}(x)) \dd\xi_y(x) \dd\check\gamma(y,y') \\
  \le \int \bar\omega_\pi(d(y,y')) \dd\check\gamma(y,y')
  \le \bar\omega_\pi\Big( \int d(y,y') \dd\check\gamma(y,y')\Big) = \bar\omega_\pi\big(\wass(\pi_*\mu_0, \check\mu_1)\big).
\end{multline*}

We now apply this with $\mu_0=T_*^\ell\nu$ and $\check\mu_1=\check\mu$: there exist $\mu_1\in \pi_*^{-1}(\check\mu)$ such that $\wass(T_*^\ell\nu,\mu_1)\le \bar\omega_\pi\big(\wass(\pi_*T_*^\ell\nu,\check\mu)\big)$. Since $\pi_*T_*^\ell\nu = S_*^\ell\pi_*\nu$ and $T_*^k\mu=\mu$, we get
\begin{align*}
\wass(T_*^{k+\ell} \nu,\mu) 
  &\le \wass(T_*^{k+\ell}\nu,T_*^k\mu_1) + \wass(T_*^k\mu_1,\mu) \\
  &\le \omega_{T^k}\big(\wass(T_*^\ell\nu,\mu_1) \big) + \bar a_k \\
  &\le \omega_{T^k} \circ \bar\omega_\pi\big(\wass(S_*^\ell\pi_*\nu,\check\mu)\big) + \bar a_k.
\end{align*}
\end{proof}

\begin{corollary}\label{coro:StableLeaf}
If $T$ is continuous, $\pi$ induces a continuous fibration, and fibers are shrunk on average with respect to $\check\mu$, then $\stabL(\mu,T_*) = \pi_*^{-1}\big(\stabL(\check\mu,S_*)\big)$.
\end{corollary}

\begin{proof}
If $\nu\in \stabL(\mu,T_*)$, then $S_*^n\pi_*\nu=\pi_*T_*^n\nu\to \pi_*\mu = \check\mu$, so that $\nu\in \pi_*^{-1}(\stabL(\check\mu,S_*))$.

Assume now that $\nu\in \proba(X)$ is such that $S_*^n\pi_*\nu\to\check\mu$ and let $\varepsilon>0$. Choose $k$ such that $a_k\le \varepsilon/2$; then there exists $\eta>0$ such that $r\in [0,\eta]\implies \omega_{T^k}\circ\bar\omega_\pi(r) \le \varepsilon/2$. Choose $\ell_0$ such that for all $\ell\ge \ell_0$, $\wass(S_*^\ell\pi_*\nu,\check\mu)\le \eta$ and apply Lemma \ref{lemm:stabL}: for all $n\ge k+\ell_0$, we have $\wass(T_*^n\nu,\mu)\le \varepsilon$.
\end{proof}

\begin{proof}[Proof of Theorem \ref{theo:mainSG}]
According to the statement to be proved, we assume that $S$ has a conformal measure $\lambda_Y$ and that the corresponding transfer operator $\chop{L}$ has a spectral gap on some Banach space of functions $(\Banach,\lVert\cdot\rVert)$ (see definition \ref{defi:SG}), with eigenfunction $h$ (normalized by $\lambda_Y(h)=1$).
Let $\nu\in\proba(X)$ such that $\pi_*\nu = f\dd\lambda_Y$ with $f\in \Banach$, and observe that $\lambda_Y(f)=1$, so that we can write $f = h+\bar f$ where $\lambda_Y(\bar f)=0$. We have $S_*^n(f\dd\lambda_y) = (h + \chop{L}^n \bar f) \dd\lambda_Y$ and 
\[\lVert \chop{L}^n \bar f \rVert_{L^1(\lambda_y)} \le  \lVert\chop{L}^n \bar f \rVert \le C(1-\delta)^n\]
for some $\delta\in(0,1)$. Since $\diam Y\le 1$, the Wasserstein metric is not greater than the total variation distance (take a transport plan that leaves the common mass in place, and moves the remaining mass arbitrarily), so that
\[\wass(h\dd\lambda_Y,(h+\chop{L}^n\bar f)\dd\lambda_Y) \le \lVert \chop{L}^n(\bar f) \dd\lambda_Y \rVert_{TV} = \lVert \chop{L}^n \bar f\rVert_{L^1(\lambda_Y)} \le C(1-\delta)^n.\]

By hypothesis there is some $\theta\in(0,1)$ such that $a_n\le C \theta^n$ for all $n\in\mathbb{N}$. Applying Lemma \ref{lemm:stabL} and denoting by $L$ the Lipschitz constant of $T$ and by $\bar\omega_\pi(r) =: K r^\alpha$ the modulus of continuity of the fibration induced by $\pi$, we get for all $k,\ell\in\mathbb{N}$:
\begin{equation}\wass(T_*^{k+\ell}\nu, \mu) \le L^k K\big(\wass(h\dd\lambda_Y,(h+\chop{L}^\ell\bar f)\dd\lambda_Y)\big)^\alpha + a_k \le  C L^k (1-\delta)^{\alpha \ell} + C\theta^k.
\label{eq:stabL}
\end{equation}
Take $\beta\in(0,1)$ such that 
\[ \beta<\frac{\alpha\log\frac{1}{1-\delta}}{\log L+\alpha\log\frac{1}{1-\delta}}\]
and define two integer sequences such that $k_n=\beta n +O(1)$, $\ell_n=(1-\beta)n+O(1)$ and $n=k_n+\ell_n$. Applying \eqref{eq:stabL} yields
$\wass(T_*^n\nu,\mu) \le C \eta^n$ for some $\eta\in(0,1)$.
\end{proof}

\section{Preserved properties of lifted invariant measures}\label{sec:properties}

\begin{assumption}
From now on the map $T$ is assumed to be an extension of $S$ with shrinking fibers.
\end{assumption}
This assumption shall remain active until the end of the article, and we shall only restate it when we want to specify the rate of shrinking or for the most important results.

With the uniqueness of the $T$-invariant lift of each $S$-invariant measure comes naturally the problem of which special properties of invariant measures are preserved under lifting (we shall later be specifically concerned with statistical properties). Theorem \ref{theo:mainLift} is the concatenation of Theorem \ref{theo:lift} with the main results of the present Section. We shall use the material of preliminary subsections \ref{sec:def-disintegration}, \ref{sec:def-physicality} and \ref{sec:def-transfer}.

\subsection{Ergodicity and mixing}\label{sec:mixing}

It is known that ergodicity is preserved by the lift map $\pi^*$, see \cite{araujo2009singular} and \cite{butterley2017disintegration}. We give an alternative proof, taking advantage of uniqueness in Theorem \ref{theo:lift}.

\begin{proposition}\label{prop:ergodic}
For all $\mu\in\proba_T(X)$, $\mu$ is ergodic if and only if $\pi_*\mu$ is ergodic.
\end{proposition}

\begin{proof}
This follows from $\pi_*$ being an affine map inducing a homeomorphism $\proba_T(X)\to\proba_S(Y)$ (Corollary \ref{coro:homeo}), since ergodic measures are the extremal points of the convex set of invariant measures.

Assume indeed $\mu$ is not ergodic: then it can be written $\mu= p\mu_0+(1-p) \mu_1$ where $\mu_0\neq\mu_1\in\proba_T(X)$ and $p\in(0,1)$. The three measures $\pi_*\mu$, $\pi_*\mu_0$ and $\pi_*\mu_1$ are $S$-invariant and satisfy $\pi_*\mu = p\pi_*\mu_0+(1-p)\pi_*\mu_1$. Moreover, $\pi_*\mu_0\neq\pi_*\mu_1$ and thus $\pi_*\mu$ is not ergodic.
If $\pi_*\mu$ is not ergodic, then similarly a decomposition lifts and $\mu$ is not ergodic either.
\end{proof}

\begin{proposition}\label{prop:weakly}
A measure $\mu\in\proba_T(X)$ is weakly mixing if and only if $\pi_*\mu$ is.\end{proposition}

\begin{proof}
That $\check\mu:=\pi_*\mu$ is weakly mixing is equivalent to $\check\mu\otimes\check\mu$ being ergodic for the diagonal action $S\times S$ on $Y\times Y$ (see e.g. \cite{walters1982introduction}, Theorem 1.24).

The map $T\times T$ is an extension of $S\times S$ with factor map $\pi\times \pi : X\times X \to Y\times Y$ and fibers $(\pi\times \pi)^{-1}(y_0,y_1)=\pi^{-1}(y_0) \times \pi^{-1}(y_1)$. If we endow products with the $\ell^\infty$ combined metric, e.g. $d((x_0,x_1),(x'_0,x'_1)) := \max(d(x_0,x'_0), d(x_1,x'_1))$, then the diameter of $(T\times T)^n([(x,x'])$ (where we recall that $[\cdot]$ denotes fibers) is the maximum diameter of $T^n([x])$, $T^n([x'])$, so that  $T\times T$ is an extension of $S\times S$ with shrinking fibers. 

By Proposition \ref{prop:ergodic}, the ergodicity of $\check\mu\otimes\check\mu$ is equivalent to the ergodicity of its lift $\mu\otimes \mu$, and thus $\check\mu$ is weakly mixing if and only if $\mu$ is.
\end{proof}

We now turn to strong mixing; recall that $\mu$ is said to be \emph{strongly mixing} if for all $f,g\in L^2(\mu)$, $\corr{n}{\mu}(f,g) \to 0$ as $n\to\infty$. We shall relate observables $f,g: X\to\mathbb{R}$ to observables on $Y$, which amounts to construct observables that are constant along fibers. As stressed in \cite{butterley2017disintegration}, a natural solution is to use average along the disintegration $(\xi_y)_{y\in Y}$ of $\mu$ with respect to $\check \mu$ (the Disintegration Theorem is recalled above as Proposition \ref{prop:disintegration}). Given a Borel function $f: X\to \mathbb{R}$, we define $\xi(f) : Y \to \mathbb{R}$ by $\xi(f)(y) = \xi_y(f)$ and $\tilde f=\xi(f)\circ \pi$. In this way, $\tilde f$ is an observable on $X$ which is constant on fibers; moreover 
\[\int \tilde f \dd\mu = \int \xi(f) \dd\check\mu = \int f \dd\mu.\] 
By convexity, $\xi(f)^p \le \xi(f^p)$, so that  $f \in L^p(\mu) \implies \xi(f) \in L^p(\check\mu)$ for all $p\in[1,\infty]$.
It is obvious that for all $u\in L^p(\check\mu)$ and $v\in L^{p'}(\check\mu)$ (where $1/p+1/p'=1$, possibly $\{p,p'\}=\{1,\infty\}$), we have
$\corr{n}{\mu}(u\circ \pi, v\circ \pi) = \corr{n}{\check\mu}(u,v)$. We will need a slightly stronger observation.
\begin{lemma}\label{lemm:CorrLift}
If $u\in L^p(\check\mu)$ and $g\in L^{p'}(\mu)$, then
$\corr{n}{\mu}(u\circ \pi, g) = \corr{n}{\check\mu}(u,\xi(g))$.
\end{lemma}

\begin{proof}
Since for all $y\in Y$, $\xi_y$ is supported on $\pi^{-1}(y)$, we have
\[
\int u\circ\pi \cdot g \dd\mu 
  = \int \Big( \int u \circ \pi(x) \cdot g(x) \dd\xi_y(x) \Big) \dd\check\mu(y) 
  = \int u \cdot \xi(g) \dd\check\mu.
\]
Applying this to $u\circ\pi\circ T^n=u\circ S^n \circ \pi$ we get the desired result.
\end{proof}

The next lemma is inspired by \cite{avila2006exponential} (Lemma 8.2), and shall be used immediately to prove that the strong mixing property lifts to extensions with shrinking fibers, and reused later to study rates of decay of correlations.
\begin{lemma}\label{lemm:CorrBound}
Assume that the fibers are $(a_n)_n$-shrinking, let $\mu\in\proba_T(X)$ and $\check\mu=\pi_*\mu$ and let $f,g:X\to \mathbb{R}$ be two observables with $f$ continuous of modulus $\omega$ and $g\in L^1(\mu)$. For all $k,m\in\mathbb{N}$ we have
\[\corr{k+m}{\mu}(f,g) \le \corr{m}{\check\mu}\big(\xi(f\circ T^k),\xi(g) \big) + \omega(a_k)\lVert g\rVert_{L^1(\mu)}.\]
\end{lemma}

\begin{proof}
Up to adding a constant we assume $\mu(f)=0$. For each $y\in Y$, 
\[\sup_{\pi^{-1}(y)} f\circ T^{k} -\inf_{\pi^{-1}(y)} f\circ T^{k} < \omega(a_k).\]
After integration with respect to $\xi_y$, we obtain that $(f\circ T^{k})\etilde := \xi(f\circ T^k)\circ\pi$ and $f\circ T^{k}$ are $\omega(a_k)$-close in the uniform norm, so that
\begin{equation*}
\corr{k+m}{\mu}(f,g) = \corr{m}{\mu}(f\circ T^{k},g)
  \le \corr{m}{\mu}((f\circ T^k)\etilde,g) + \omega(a_k) \lVert g\rVert_{L^1(\mu)}.
\end{equation*}
Applying Lemma \ref{lemm:CorrLift} we get $\corr{m}{\mu}((f\circ T^k)\etilde,g) = \corr{m}{\check\mu}\big(\xi(f\circ T^k),\xi(g)\big)$.
\end{proof}

\begin{proposition}
A measure $\mu\in\proba_T(X)$ is strongly mixing if and only if $\pi_*\mu$ is.
\end{proposition}

\begin{proof}
Assume first that $\mu$ is strongly mixing, and let $u,v:Y\to \mathbb{R}$ be observables in $L^2(\check\mu)$. Since $\corr{n}{\mu}(u\circ \pi,v\circ \pi) = \corr{n}{\check\mu}(u,v)$ and $\mu$ is strongly mixing, this goes to $0$ as $n$ goes to $\infty$. (This is classical and does not use the shrinking property).

Assume now that $\check\mu$ is strongly mixing. Given $f,g\in L^2(\mu)$, define $\xi(f), \tilde f, \xi(g),\tilde g$ as above and recall that $\xi(f),\xi(g)\in L^2(\check\mu)$. Fix $\varepsilon>0$ and let $h$ be a continuous approximation of $f$, with $\lVert f-h\rVert_{L^2(\mu)}<\varepsilon$. We have
\begin{align*}
\Big\lvert \int (f-h)\circ T^n \cdot g \dd\mu \Big\rvert 
  &\le \Big(\int (f-h)^2\circ T^n \dd\mu \Big)^{\frac12} \Big(\int g^2 \dd\mu\Big)^{\frac12} \\
  &\le \Big(\int (f-h)^2 \dd\mu \Big)^{\frac12}\Big(\int g^2 \dd\mu\Big)^{\frac12} 
  &\le   \varepsilon \lVert g\rVert_{L^2(\mu)}
\end{align*} 
so that
\begin{equation}\corr{n}{\mu}(f,g) \le \corr{n}{\mu}(h,g) + \varepsilon \lVert g\rVert_{L^2(\mu)}.
\label{eq:mix1}
\end{equation}

Let $\omega$ be a modulus of continuity of $h$, and let $(a_n)_n$ be a shrinking sequence. There is a $k$ such that $\omega(a_{k})<\varepsilon$. 
By Lemma \ref{lemm:CorrBound} and using $\lVert\cdot\rVert_{L^1(\mu)} \le \lVert\cdot\rVert_{L^2(\mu)}$,
\begin{equation}
\corr{k+m}{\mu}(h,g) \le \corr{m}{\check\mu}\big( \xi(h\circ T^k), \xi(g) \big) + \varepsilon \lVert g\rVert_{L^2(\mu)}
\label{eq:mix2}
\end{equation}

Combining \eqref{eq:mix1} and \eqref{eq:mix2} we get $\corr{k+m}{\mu}(f,g) \le \corr{m}{\check\mu}\big( \xi(h\circ T^k),\xi(g) \big) + 2\varepsilon\lVert g\rVert_{L^2(\mu)}$. Since $\check\mu$ is strongly mixing, there is an $m_0$ such that for all $n>m_0+k$, $\corr{n}{\mu}(f,g) \le (1+2\lVert g\rVert_{L^2(\mu)})\varepsilon$, and $\mu$ is strongly mixing.
\end{proof}

\subsection{Entropy}

Entropy preservation in Theorem \ref{theo:mainLift} is unsurprising and, thanks to the uniqueness in Theorem \ref{theo:lift}, follows easily from the relative variational principle established by Ledrappier and Walters \cite{ledrappier1977relativised}: for all $\check\mu\in\proba_S(Y)$,
\[\sup_{\mu\in\pi_*^{-1}(\check\mu)} \entKS(T,\mu) = \entKS(S,\check\mu) + \int h(T,\pi^{-1}(y)) \dd\check\mu(y)\]
where $\entKS(T,\mu)$ is the Kolmogorov-Sinai entropy and $h(T,K)$ is the topological entropy of $T$ on the (non-necessarily invariant) compact set $K\subset X$.

\begin{proposition}\label{prop:entropy}
If $S,T$ are continuous, then $h(T,\mu) = h(S,\pi_*\mu)$ for all $\mu\in\proba_T(X)$.
\end{proposition}

\begin{proof}
Let $\mu\in\proba_T(X)$ and $\check\mu=\pi_*\mu$. By Theorem \ref{theo:lift}, $\pi_*^{-1}(\check\mu)=\{\mu\}$, so that the Ledrappier-Walters relative variational principle reads $\entKS(T,\mu) = \entKS(S,\check\mu) + \int \ent(T,\pi^{-1}(y)) \dd\check\mu(y)$, and we are left with proving $\ent(T,\pi^{-1}(y)) \equiv 0$.

Let $y\in Y$ and $\delta>0$. There is an $n_0\in\mathbb{N}$ such that for all $n>n_0$, $a_n < \delta$. Let $E_0$ be a maximal $(n_0,\delta)$-separated set of $\pi^{-1}(y)$. For all $x,x'\in E_0$ and all $m$ such that $n_0<m\le n$, $d(T^m x,T^m x')\le a_m < \delta$ so that $E_0$ must be $(n,\delta)$-separated as well. It follows that the cardinal of an $(n,\delta)$-separated set is bounded independently of $n$, and therefore $\ent(T,\pi^{-1}(y))=0$.
\end{proof}

\subsection{Absolute continuity, physicality, observability}

Assume here that $X$ and $Y$ are equipped with reference measures $\lambda_X$ and $\lambda_Y$. In general, the lift to an extension with shrinking fibers of an absolutely continuous invariant probability (Acip) is not itself an Acip; the map $S$ could have an Acip while $T$ does not (e.g. take $Y$ to be a point, $T$ contracting). However the weaker property of physicality is preserved under a mild regularity assumption on $\pi$.

\begin{proposition}\label{prop:physical}
Assume that $T$ is continuous and $X,Y$ are equipped with reference measures with respect to which $\pi$ is non-singular. A measure $\mu\in\proba_T(X)$ is physical if and only if $\pi_*\mu$ is.
\end{proposition}

\begin{proof}
As usual we set $\check\mu=\pi_*\mu$.
For all $x\in X$ we have $\pi_*\big(\frac1n \sum_{k=0}^{n-1} \delta_{T^kx} \big) = \frac1n \sum_{k=0}^{n-1} \delta_{S^k\pi(x)}$.
If $\frac1n \sum_{k=0}^{n-1} \delta_{T^kx}$ converges to some $\nu\in\proba_T(X)$, then $\frac1n \sum_{k=0}^{n-1} \delta_{S^k\pi(x)}$ converges to $\pi_*\nu$. This proves that $\basin(\mu)\subset \pi^{-1}(\basin(\check\mu))$; we get equality by compactness: if $\pi(x)\in \basin(\check\mu)$, then any cluster point of the sequence $\big(\frac1n \sum_{k=0}^{n-1} \delta_{T^kx}\big)_n$ is mapped by $\pi_*$ to $\check\mu$. Since $T$ is continuous, such cluster points are $T$-invariant, so that Theorem \ref{theo:lift} implies that $\mu$ is the unique cluster point of the sequence, hence its limit.

Since $\pi$ is non-singular, $\lambda_Y(\basin(\check\mu))>0$ if and only if $\lambda_X(\pi^{-1}(\basin(\check\mu)))>0$, i.e.
$\check\mu$ is physical if and only if $\mu$ is physical.
\end{proof}

Since ergodic Acip are particular cases of physical measures, while they do not necessarily lift to Acips, they do lift to physical measures. This implies that many weakly hyperbolic systems have physical measures (see e.g. Corollaries \ref{coro:mainPM}, \ref{coro:mainlog}).

\begin{remark}\label{r:counter-ex}
The assumption that $T$ is continuous cannot be lifted, as the following example shows. Let $X=[0,1]$ and consider a map $T$, continuous on $(0,1]$, such that $T(x)<x$ for all $x\in(0,1)$ and $T(1)=T(0)=1$, as pictured in Figure \ref{fig:counter-ex}. Let $Y=[0,1]/(0\sim 1)$ be the circle and $\pi:X\to Y$ be the usual projection; set $\bar 0=\pi(0)=\pi(1)$. Let $\lambda_X$ and $\lambda_Y=\pi_*\lambda_X$ be the Lebesgue measures. Most fibers of $\pi$ are reduced to a single point, the only exception being $\pi^{-1}(\bar 0) = \{0,1\}$. Since $T(\{0,1\})=\{1\}$, $T$ has shrinking fibers in a very strong sense, with $a_n=0$ for all $n>0$. 

On the one hand $S$ is a ``parabolic'' circle map: it has a unique fixed point $\bar0$, which is attractive on one side and repulsive on the other, and every orbit converges to $\bar 0$. In particular, $\delta_{\bar 0}$ is the unique invariant measure of $S$ and is physical: its basin is the whole of $Y$. On the other hand, $\delta_1$ is the unique invariant measure of $T$ but it is not physical, since for all $x\in(0,1)$ we have $T^k(x)\to 0$ and $\frac1n\sum_{k=0}^{n-1}\delta_{T^k x}\to \delta_0$.
\end{remark}

\begin{figure}
\begin{center}
\includegraphics[scale=1]{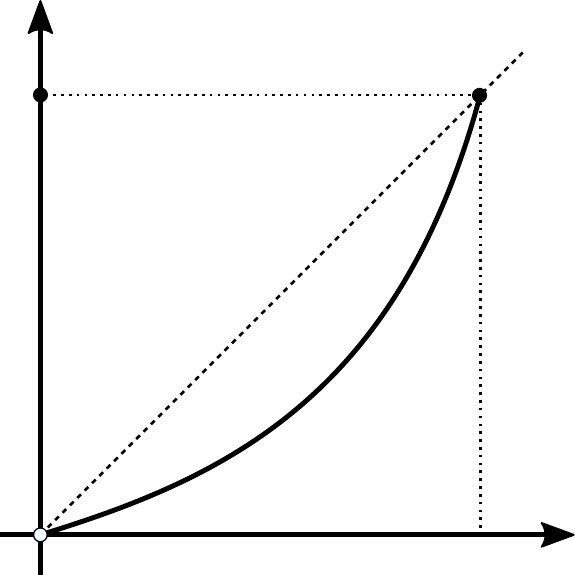}
\caption{A discontinous extension $T$ of a circle map $S$, where $S$ has a physical measure but $T$ has none.}\label{fig:counter-ex}
\end{center}
\end{figure}

Finally, we show that it is not much more difficult to lift observability.
\begin{proposition}\label{prop:observable}
Assume that $T$ is continuous and $X,Y$ are equipped with reference measures with respect to which $\pi$ is non-singular. A measure $\mu\in\proba_T(X)$ is observable if and only if $\pi_*\mu$ is.
\end{proposition}

\begin{proof}
Given any $\varepsilon>0$, since $\pi_*$ is continuous there exist some $\eta>0$ such that for all $\mu_0,\mu_1\in\proba(X)$, $\wass(\mu_0,\mu_1)<\eta \implies \wass(\pi_*\mu_0,\pi_*\mu_1)<\epsilon$. Let $x\in\basin_\eta(\mu)$: there exist $\mu_0\in\proba_T(X)$ such that $\wass(\mu_0,\mu)<\eta$ and an increasing sequence of positive integers $(n_k)_{k\in\mathbb{N}}$ such that
$\mu_0 = \lim_k \frac{1}{n_k} \sum_{j=0}^{n_k-1} \delta_{T^j(x)}$. Then $\frac{1}{n_k} \sum_{j=0}^{n_k-1} \delta_{S^j\pi(x)}\to \pi_*\mu_0\in\proba_S(Y)$ and $\wass(\pi_*\mu_0,\check\mu) \le \varepsilon$, so that $\pi(x) \in \basin_\varepsilon(\check\mu)$. We have proved $\basin_\eta(\mu)\subset \pi^{-1}(\basin_\varepsilon(\check\mu))$; if $\mu$ is observable, then $\lambda_X(\basin_\eta(\mu))>0$ and by non-singularity of $\pi$, we deduce that $\check\mu$ is observable.

Since $\pi^*$ is continuous, for all $\varepsilon>0$ there exist an $\eta>0$ such that for all $\nu_0,\nu_1\in\proba_S(Y)$, $\wass(\nu_0,\nu_1)<\eta \implies \wass(\pi^*\nu_0,\pi^*\nu_1)<\epsilon$.
Let $x\in X$ such that $\pi(x)\in \basin_\eta(\check\mu)$. There exist $\nu_0\in\proba_S(Y)$ such that $\wass(\check\mu,\nu_0)<\eta$ and an increasing sequence of positive integers $(n_k)_{k\in\mathbb{N}}$ such that
$\nu_0 = \lim_k \frac{1}{n_k} \sum_{j=0}^{n_k-1} \delta_{S^j\pi(x)} =\lim_k \pi_*\big( \frac{1}{n_k} \sum_{j=0}^{n_k-1} \delta_{T^j x} \big)$.
It follows that any cluster point of $\big(\frac{1}{n_k} \sum_{j=0}^{n_k-1} \delta_{T^jx}\big)_k$ is mapped by $\pi_*$ to $\nu_0$. Since $T$ is continuous, these cluster points are $T$-invariant and there is only one of them, $\pi^*\nu_0$. Since $\wass(\nu_0,\check\mu)<\eta$, we moreover have $\wass(\pi^*\nu_0,\pi^*\check\mu)<\varepsilon$, so that $x\in \basin_\varepsilon(\mu)$.
We thus proved that $\pi^{-1}(\basin_\eta(\check\mu))\subset \basin_\varepsilon(\mu)$, from which we deduce that if $\check\mu$ is observable, then so is $\mu$.
\end{proof}

\section{Equilibrium states and statistical properties}\label{sec:equilibrium}

In this Section we consider some classical objects and properties that form the core of the Thermodynamical Formalism, and lift them from $S$ to $T$. This will for example be used to recover information about certain hyperbolic maps from information about expanding maps. This is an old strategy, notably well developed in symbolic dynamics, that have been extended more generally through Markov partition and coding. More recently, a ``direct lifting'' approach has been used frequently, often in quite specific cases. Our goal is to use this approach in the most general way while keeping all proofs simple. We shall use preliminary subsections \ref{sec:def-moduli}, \ref{sec:def-disintegration} (definition of a section), \ref{sec:def-stat}.

\subsection{Equilibrium states}

Let $\varphi:X\to\mathbb{R}$ be a function, here called an \emph{potential}, to be interpreted physically (up to the sign) as a density of energy: a $T$-invariant measure $\mu$ is called a ``state'', the total energy of the system in state $\mu$ being $-\mu(\varphi)$. The ``free energy'' is then $\mathscr{F}(\mu) := \entKS(T,\mu)+\mu(\varphi)$, and we seek \emph{equilibrium states}, i.e. invariant measures maximizing free energy. The main questions underlying the ``thermodynamical formalism'' are existence, uniqueness, and statistical properties of equilibrium states. 

Here of course we want to relate this to the corresponding situation for $S$; since $\check\mu = \pi_*\mu$ is the state on $Y$ corresponding to $\mu$ and $\entKS(S,\check\mu) = \entKS(T,\mu)$ (Proposition \ref{prop:entropy}), one only needs to consider the energy term. We would thus like to construct a potential $\check\varphi : Y\to\mathbb{R}$ related to $\varphi$; using the disintegration of $\mu$ to construct $\xi(\varphi)$ as in Section \ref{sec:mixing} is not suitable here since invariant measures are to be considered all at once and compared. We will rather add a suitable \emph{coboundary} to $\varphi$, as is classically done in the case of shifts, see \cite{bowen1975equilibrium}.

Coboundaries are defined as the functions of the form $h-h\circ T : X\to\mathbb{R}$. They are important because for \emph{all} $T$-invariant measure $\mu$, we have $\mu(h-h\circ T) = \mu(h) - \mu(h\circ T) = 0$: adding a coboundary to a potential does not change its energy with respect to any state. We will construct a potential $\hat\varphi = \varphi + h-h\circ T$ that is constant on fibers (then $\hat\varphi = \check\varphi \circ \pi$ will define $\check\varphi:Y\to\mathbb{R}$).

\begin{lemma}\label{lemm:coboundary}
Assume that fibers are $(a_n)_n$-shrinking, and that $\pi$ admits a continuous section $\sigma : Y\to X$.
Let $\varphi : X\to\mathbb{R}$ be a $\omega$-continuous potential where  $\sum_{n\ge 0} \omega(a_n) <\infty$, and set
$h = \sum_{n=0}^\infty \big( \varphi T^n\sigma \pi - \varphi T^n \big) $.
Then  $h:X\to\mathbb{R}$ is well-defined and $\hat\varphi := \varphi +h- h\circ T$ is constant on each fiber. 

If $T$ is continuous, so is $h$. If $T$ is $L$-Lipschitz, then for all $x,x'\in X$ and all $N\in\mathbb{N}$:
\[\lvert h(x)- h(x')\rvert \le 2\Hol_\omega(\varphi) \Big(\sum_{n=0}^N \omega\big(L^n \omega_{\sigma\pi}(d(x,x'))\big) + \sum_{n>N} \omega(a_n) \Big)\]
where $\omega_{\sigma\pi}$ is any modulus of continuity of $\sigma\pi$ such that $\omega_{\sigma\pi}(r)\ge r$ for all $r\in[0,\diam X]$.
\end{lemma}
(The assumption on $\omega_{\sigma\pi}$ can always be obtained up to increase the modulus, and is only meant to simplify the conclusion.)

\begin{proof}
Let $H=\Hol_\omega(\varphi)$.
For all $x\in X$, $\sigma\pi(x)$ and $x$ lie on the same fiber, so that $d(T^n\sigma\pi(x),T^n(x)) \le a_n$ and $\lvert \varphi T^n\sigma \pi(x) - \varphi T^n(x) \rvert \le H\omega(a_n)$. The convergence of $\sum \omega(a_n)$ ensures the uniform convergence of the series defining $h$ which is therefore well-defined, and continuous whenever $T$ is.

Next, we have
\begin{align*}
\hat\varphi &= \varphi + \sum_{n=0}^\infty\big( \varphi T^n\sigma\pi -\varphi T^n \big) - \sum_{n=0}^\infty\big(\varphi T^n\sigma\pi T + \varphi T^{n+1} \big)\\
  &= \varphi + \varphi\sigma\pi -\varphi + \sum_{n=0}^\infty \big( \varphi T^{n+1}\sigma\pi -\varphi T^{n+1} -\varphi T^n\sigma S\pi + \varphi T^{n+1} \big) \\
  &= \varphi\sigma\pi + \sum_{n=0}^\infty \big( \varphi T^{n+1}\sigma\pi -\varphi T^n\sigma S\pi  \big)
\end{align*}
which is constant on fibers since $\pi$ factors on the right.
Assume now that $T$ is $L$-Lipschitz and $\sigma\pi$ has modulus of continuity $\omega_{\sigma\pi}$. Then
\begin{align*}
\lvert h(x)- h(x')\rvert 
  &\le \sum_{n=0}^\infty \big\lvert \varphi T^n\sigma \pi(x) - \varphi T^n(x) - \varphi T^n\sigma \pi(x') + \varphi T^n(x')\big\rvert \\
  &\le \sum_{n=0}^N \big\lvert  \varphi T^n\sigma \pi(x) - \varphi T^n\sigma \pi(x') \big\rvert + \sum_{n=0}^N \big\lvert \varphi T^n(x)- \varphi T^n(x') \big\rvert + \\
&\qquad  \sum_{n>N} \big\lvert \varphi T^n\sigma \pi(x) - \varphi T^n(x) \big\rvert + \sum_{n>N} \big\lvert \varphi T^n\sigma \pi(x') - \varphi T^n(x')\big\rvert \\
  &\le \sum_{n=0}^N H\omega\big(L^n \omega_{\sigma\pi}(d(x,x'))\big) + \sum_{n=0}^N H \omega\big(L^n d(x,x')\big) + 2\sum_{n>N} H \omega(a_n).
\end{align*}
We conclude by using $d(x,x')\le \omega_{\sigma\pi}(d(x,x'))$.
\end{proof}

\begin{theorem}\label{theo:coboundary}
Assume that $T$ is an extension of $S$ with $(a_n)_n$-shrinking fibers, that $\pi$ has modulus of continuity $\omega_\pi(r)\ge r$ and admits a Lipschitz section $\sigma : Y\to X$, and that $T$ is $L$-Lipschitz. Let $\omega$, $\check\omega$ be two moduli of continuity with $\check\omega\gtrsim \omega$.  

If for some constant $D>0$ and for all $r\in[0,\diam X]$ there exist some $N=N(r)\in\mathbb{N}$ such that
\begin{equation}  \sum_{n>N} \omega(a_n) \le D\check\omega(r) \quad\text{and}\quad \sum_{n=0}^N \omega\big(L^n \omega_\pi(r)\big) \le D \check\omega(r),
\label{eq:coboundary}
\end{equation}
then:
\begin{enumerate}
\item for all $\omega$-continuous potential $\varphi : X\to\mathbb{R}$ there is an $\check\omega$-continuous potential $\check\varphi: Y\to\mathbb{R}$ such that $\varphi$ differs from $\hat\varphi = \check\varphi\circ\pi$ by a coboundary,
\item for all $\mu\in \proba_T(X)$, writing $\check\mu = \pi_*\mu$ we have
\[\entKS(T,\mu) + \mu(\varphi) = \entKS(S,\check\mu)+ \check\mu(\check\varphi);\]
in particular $\pi^*$ realizes a bijection between equilibrium states of $\check\varphi$ and equilibrium states of $\varphi$,
\item we can realize $\varphi\mapsto \check\varphi$ as a continuous linear map from $\Hol_\omega(X)$ to $\Hol_{\check\omega}(Y)$.
\end{enumerate} 
\end{theorem}

\begin{proof}
Given $\varphi\in\Hol_\omega(X)$, let $h$ be the function defined by Lemma \ref{lemm:coboundary}. The hypotheses are taylored to ensure that $h$ is $\check\omega$-continuous (more precisely $\Hol_{\check\omega}(h)\le 4D\Hol_\omega(\varphi)$).
It follows that $\Hol_{\check\omega}(h\circ T) \le 4LD \Hol_\omega(\varphi)$, and the potential $\hat\varphi = \varphi + h-h \circ T$ is $\check\omega$-continuous. Since $\sigma$ is Lipschitz, $\check\varphi :=\hat\varphi\circ \sigma$ is also $\check\omega$-continuous. Since $\hat\varphi$ is constant on fibers, $\check\varphi\circ\pi = \hat\varphi$.

The equality of free energies follows from the equality of entropies (Proposition \ref{prop:entropy}) and from $\mu(\varphi) = \mu(\hat\varphi) = \pi_*\mu(\check\varphi)$. 

The fact that $\varphi\mapsto \check\varphi$ is continuous linear follows from the construction.
\end{proof}

From here the game consists in finding the optimal choice of $N(r)$ depending on the available assumptions. We will restrict in the following to the case when $\pi$ is H\"older continuous, a common situation in hyperbolic dynamics.
\begin{corollary}\label{coro:coboundary}
Assume that $T$ is $L$-Lipschitz for some $L\ge 1$, that $\sigma$ is Lipschitz and that $\pi$ is $\beta$-H\"older.
\begin{enumerate}
\item\label{enumi:coboundary1} If the fibers are exponentially shrinking with ratio $\theta\in(0,1)$ and $\omega=\hol{\alpha}$ is a H\"older modulus of continuity, then the conclusions of Theorem \ref{theo:coboundary} hold with  $\check\omega = \hol{\gamma}$ where $\gamma = \frac{\alpha\beta}{1-\log L/\log\theta}$.
\item\label{enumi:coboundary2} If the fibers are polynomially shrinking with degree $d>0$, and $\omega = \hol{\alpha}$ where $\alpha>1/d$, then the conclusions of Theorem \ref{theo:coboundary} hold with $\check\omega = \holl{\alpha'}$ where $\alpha'=\alpha d-1$.
\item\label{enumi:coboundary3} If the fibers are exponentially shrinking and $\omega=\holl{\alpha}$ with $\alpha>1$, then the conclusions of Theorem \ref{theo:coboundary} hold with $\check\omega = \holl{\alpha'}$ where $\alpha'=(\alpha-1)/2$. \end{enumerate}
\end{corollary}

Note that we can always replace $\check\omega$ with a larger modulus if needed. In particular, in the case of exponentially shrinking fibers, if each H\"older continuous potential on $Y$ has a unique equilibrium state for $S$, then each H\"older continuous potential on $X$ has a unique equilibrium state for $T$.

\begin{proof}
We apply Theorem \ref{theo:coboundary} three times.
For \ref{enumi:coboundary1}, take $N(r)=\frac{\gamma\log r}{\alpha\log\theta}+O(1)$: then 
\[\sum_{n>N}\omega(a_n) \lesssim \theta^{\alpha N} \lesssim r^\gamma \quad\text{and}\quad 
\sum_{n=0}^N \omega(L^n \omega_\pi(r)) \lesssim L^{\alpha N} r^{\alpha\beta} \lesssim r^{\gamma'}\]
with $\gamma'=\gamma\frac{\log L}{\log \theta}+\alpha\beta = \gamma$.

For \ref{enumi:coboundary2}, take $N(r)=\eta\log\frac{r_{\alpha'}}{r} / \log L + O(1)$ with any $\eta<\beta$: then
\[\sum_{n>N}\omega(a_n) \lesssim \frac{1}{N^{\alpha d-1}} \lesssim \frac{1}{ \big(\log \frac{r_{\alpha'}}{r}\big)^{\alpha'}}\]
and 
\[\sum_{n=0}^N \omega(L^n \omega_\pi(r)) \lesssim L^{\alpha N} r^{\alpha\beta} \lesssim r^{\alpha(\beta-\eta)}\ll \holl{\alpha'}(r).\]

For \ref{enumi:coboundary3}, take $N(r)=(\log\frac{r_{\alpha'}}{r})^{\frac12} + O(1)$: then, using Proposition \ref{prop:alphalog} for the second term,
\[\sum_{n>N}\omega(a_n) \lesssim \frac{1}{N^{\alpha-1}} \lesssim \frac{1}{ \big(\log \frac{r_{\alpha'}}{r}\big)^{\alpha'}} \quad\text{and}\quad 
\sum_{n=0}^N \omega(L^n\omega_\pi(r)) \lesssim \frac{N^{\alpha+1}}{\big(\log\frac{r_\alpha}{r^\beta}\big)^\alpha} \lesssim \frac{1}{\big(\log\frac{r_{\alpha'}}{r}\big)^{\alpha'}}.\]
\end{proof}

\subsection{Statistical properties}

We would now like to lift statistical properties, assuming them for $\check\mu\in\proba_S(Y)$ and deducing them for its lift $\mu\in\proba_T(X)$. One can in principle lift a decay of correlations (which we will consider next) and then use it to prove statistical properties, but it is in fact simpler to use Theorem \ref{theo:coboundary} on observables to lift statistical properties directly.

\begin{proposition}\label{prop:statistical}
Assume $\pi$ has a section $\sigma:Y\to X$.
Consider $\mu\in\proba_T(X)$, $\check\mu=\pi_*\mu$, $\LimTh{T}\in\{\mathrm{LIL},\mathrm{CLT},\mathrm{ASIP}\}$, and let $\omega,\check\omega$ be two moduli of continuity. If
\begin{enumerate}
\item for each $f\in\Hol_\omega(X)$ there is a continuous $h:X\to \mathbb{R}$ such that $\hat f = f +h -h\circ T$ is constant on fibers and $\check f=\hat f\circ \sigma$ belongs to $\Hol_{\check\omega}(Y)$,
\item $\check\mu$ satisfies $\LimTh{T}$ for all $\check\omega$-continuous observables,  
\end{enumerate}
then $\mu$ satisfies $\LimTh{T}$ for all $\omega$-continuous observables, with the same parameters than $\check\mu$ (see Definition \ref{defi:stat}).
\end{proposition}

\begin{proof}
This is classical and straightforward. For all $f\in\Hol_\omega(X)$ we have
\[\sum_{k=1}^n \hat f\circ T^k = \sum_{k=1}^n f\circ T^k + h\circ T-h\circ T^{n+1} = \sum_{k=1}^n f\circ T^k + O(1)\]
where the $O(1)$ is bounded in the uniform norm, and $\sum_{k=1}^n \hat f\circ T^k = \big(\sum_{k=1}^n \check f \circ S^k \big)\circ \pi$.

Then, when $(\check\mu,\check f)$ satisfy the LIL with some variance $\sigma_{\check f}>0$, we have:
\[\frac{\sum_{k=1}^n f\circ T^k(x) -n \mu(f)}{\sqrt{2 n\log \log n}} = \frac{\sum_{k=1}^n \check f\circ S^k(\pi(x)) -n \check\mu(\check f)}{\sqrt{2 n\log \log n}} + o(1),\]
and the superior limit is $\sigma_{\check f}$ for all $x\notin \pi^{-1}(E)$ for some $\check\mu$-negligible set $E$. Since $\check\mu=\pi_*\mu$, $\pi^{-1}(E)$ is $\mu$-negligible.

In the case of the CLT, for all $\varepsilon>0$, for all $n$ large enough
\[\Big\lVert\sum_{k=1}^n \hat f\circ T^k -  \sum_{k=1}^n f\circ T^k \Big\rVert_\infty \le \varepsilon \sqrt{n}\]
and therefore
\begin{align*}
\mu\Big\{x\in X : \frac1{\sqrt{n}} \sum_{k=1}^n f\circ T^k(x) \le r \Big\} 
  &\ge \mu\Big\{x\in X : \frac1{\sqrt{n}} \sum_{k=1}^n \hat f\circ T^k(x) \le r-\varepsilon \Big\} \\
  &= \check\mu\Big\{y\in Y : \frac1{\sqrt{n}} \sum_{k=1}^n \check f\circ S^k(y) \le r-\varepsilon \Big\} \\
  &\to G_{\check\mu(\check f),\sigma_{\check f}}(r-\varepsilon)
\end{align*}
so that, by continuity of $G_{\check\mu(\check f),\sigma_{\check f}}=G_{\mu(f),\sigma_{\check f}}$,
\[\liminf_{n\to\infty} \mu\Big\{x\in X : \frac1{\sqrt{n}} \sum_{k=1}^n f\circ T^k(x) \le r \Big\} \ge G_{\mu(f),\sigma_{\check f}}(r).\]
The superior limit is treated in the same way, and we get the CLT for $\mu$, with the same variance.

In the case of the ASIP, we have processes $(\check A_k)_{k\in\mathbb{N}}$, whose law is the same than that of $(\check f\circ S^k(\check Z))_{k\in\mathbb{N}}$ where $\check Z$ has law $\check\mu$, and $(B_k)_{k\in\mathbb{N}}$, independent Gaussian of mean $\check\mu(\check f)=\mu(f)$ and variance $\sigma_{\check f}$, such that $\lvert\sum_{k=1}^n \check A_k -\sum_{k=1}^n B_k\lvert = o(n^\lambda)$ almost surely. We first construct a random variable $Z$ with law $\mu$ such that $\check Z=\pi(Z)$: up to enrich $\Omega$, we can assume to have a uniform random variable $V$ on $[0,1]$ independent from all previous random variables, and by measurable selection we have a measurable family of measurable maps $\Xi_y :[0,1]\to X$ such that $\Xi_y(V)$ has law $\xi_y$, where $(\xi_y)_{y\in Y}$ is the disintegration of $\mu$ with respect to $\pi$. Then $Z=\Xi_{\check Z}(V)$ is the desired random variable; now $(\check A_k)_{k\in\mathbb{N}}$ has the same law as $(\hat f\circ T^k(Z))_{k\in\mathbb{N}}$. Since $f\circ T^k(Z) = \hat f\circ T^k(Z)-h(T^k Z)+h(T^{k+1} Z)$, there is a process $(H_k)_k$ with the same law as $(h( T^k Z)-h(T^{k+1} Z))_k$ such that $(A_k)_k=(\check A_k-H_k)_k$ has the same law as $(f\circ T^k(Z))_k$; in particular, $\sum_{k=1}^n H_k$ has the same law as $h(T Z)-h(T^{n+1} Z)$ and is bounded almost surely by $2\lVert h\rVert_\infty$. At last, almost surely
\[\big\lvert\sum_{k=1}^n A_k -\sum_{k=1}^n B_k\big\lvert \le \big\lvert\sum_{k=1}^n \check A_k -\sum_{k=1}^n B_k\big\lvert + \lvert \sum_{k=1}^n H_k \rvert = o(n^\lambda) + O(1)=o(n^\lambda).\]
\end{proof}

Theorem \ref{theo:mainstat} follows directly from Proposition \ref{prop:statistical} and Corollary \ref{coro:coboundary}. More generally, applying Theorem \ref{theo:coboundary} we obtain the following.
\begin{theorem}\label{theo:stat}
Assume that $T$ is an $L$-Lipschitz extension of $S$ with $(a_n)_n$-shrinking fibers, that the factor map $\pi$ is $\beta$-H\"older-continuous and that there is a Lipschitz section $\sigma:Y\to X$. 

We consider moduli of continuity $\omega_p, \omega_o, \check\omega_p, \check\omega_o$ where $p$ stand for ``potential'' and $o$ for ``observable'' and a limit theorem $\LimTh{T}\in\{\text{LIL},\text{CLT},\text{ASIP}\}$. 

If $S$ satisfies $\UE(\check\omega_p[\rho];\LimTh{T},\check\omega_o)$ and if for some constant $D>0$ and for all $r\in[0,\diam X]$ there exist some $N=N(r)\in\mathbb{N}$ such that for each $i\in\{o,p\}$:
\begin{equation}  \sum_{n>N} \omega_i(a_n) \le D\check\omega_i(r) \quad\text{and}\quad \sum_{n=0}^N \omega_i(L^n r^\beta) \le D \check\omega_i(r),
\label{eq:maincoboundary}
\end{equation}
then $T$ satisfies $\UE(\omega_p[C\rho];\LimTh{T},\omega_o)$ (with the same parameters in $\LimTh{T}$ than for $S$).
\end{theorem}

\section{Decay of correlations}\label{sec:decay}

The fact that strong mixing is preserved by the lift map hints to the fact that decay of correlation, which quantify mixing for regular enough observables, might also lift from $(S,\check\mu)$ to $(T,\mu)$. This Section uses the preliminary subsections \ref{sec:def-moduli}, \ref{sec:def-disintegration}, \ref{sec:def-transfer}.

Assume we have some decay of correlations for observables of a given regularity for $(S,\check\mu)$. Given $f,g:X\to\mathbb{R}$ we have $\corr{n}{\mu}(\tilde f, \tilde g)=\corr{n}{\check\mu}(\xi(f),\xi(g))$, and Lemma \ref{lemm:CorrBound} relates $\corr{n}{\mu}(f,g)$ to $\corr{n}{\mu}(\tilde f, \tilde g)$ (up to composition with $T^n$ and a shift in $n$). The crucial missing piece is to understand whether the disintegration $(\xi_y)_{y\in Y}$ preserves regularity. For example, if $f$ is H\"older does it follow that $\xi(f)$ is H\"older?

\subsection{Regularity of the disintegration}

The following is close from Proposition 3 in \cite{butterley2017disintegration}, but we do not assume a skew product structure and add continuity to the conclusion.
\begin{lemma}
Assume that $T$ is continuous and that there is a continuous section $\sigma$. Fix $\mu\in\proba_T(X)$, $\check\mu=\pi_*\mu$ and let $(\xi_y)_y$ be the disintegration of $\mu$ with respect to $\pi$ and $\chop{L}$ be the transfer operator of $(S,\check\mu)$.

Then for all $y\in Y$ and all continuous $f:X\to\mathbb{R}$ we have $\xi_y(f) = \lim_n \chop{L}^n(fT^n\sigma)(y)$. If moreover  $\chop{L}$ sends continuous functions to continuous functions, then $(\xi_y)_{y\in Y}$ preserves continuity.
\end{lemma}

\begin{proof}
For all $n,m\in\mathbb{N}$ and all continuous $f:X\to\mathbb{R}$, using $g=\chop{L}^m(g S^m)$, $\lVert \chop{L}(g)\rVert_\infty\le \rVert g\rVert_\infty$ and $\pi T^m\sigma(y) = \pi\sigma S^m(y)$ we have:
\begin{align*}
\big\lVert \chop{L}^{n+m}(fT^{n+m}\sigma) - \chop{L}^n(fT^n\sigma) \big\rVert_\infty
  &= \big\lVert \chop{L}^{n+m}(fT^n\circ T^m\sigma) - \chop{L}^{n+m}(fT^n\circ \sigma S^m) \big\rVert_\infty \\
  &\le \big\lVert fT^n\circ T^m\sigma - fT^n\circ \sigma S^m \big\rVert_\infty \\
  &\le \omega(a_n)
\end{align*}
where $\omega$ is a modulus of continuity of $f$ and $(a_n)_n$ is a shrinking sequence.

It follows that $\zeta_y^n := \chop{L}^n(fT^n\sigma)(y)$ converges as $n\to\infty$, uniformly in $y\in Y$, and that the limit $\zeta_y :=\lim_n \zeta_y^n$ defines for each $y\in Y$ a continuous linear form on $\C^0(X)$, i.e. a measure  on $X$. 

If $\chop{L}$ sends continuous functions to continuous functions, then for all $n\in\mathbb{N}$ the function $y\mapsto \zeta_y^n(f)$ is continuous, and by uniform convergence so is $y\mapsto \zeta_y(f)$.

We have left to check that $(\zeta_y)_{y\in Y}$ coincides with $(\xi_y)_{y\in Y}$ on a set of full $\check\mu$ measure. Since $\chop{L}^n(\one T^n\sigma) = \one$ and $\chop{L}^n(f T^n\sigma) \ge 0$ whenever $f\ge 0$, $\zeta_y$ is a probability measure for each $y$. If $f\equiv 0$ on $\pi^{-1}(y)$, then $f T^n\sigma\equiv 0$ on $S^{-n}(y)$ and $\chop{L}^n(f T^n\sigma)(y) = 0$, so that $\zeta_y(f)=0$; i.e. $\zeta_y$ is concentrated on $\pi^{-1}(y)$. Last, 
\[\int \zeta_y(f) \dd\check\mu  
  = \lim_n \int \chop{L}^n(f T^n\sigma)  \dd\check\mu
  = \lim_n \int f T^n\sigma \dd\check\mu
  = \int f \dd\big(\lim_n T_*^n(\sigma_*\check\mu)\big)
  = \int f \dd\mu 
\]
and by uniqueness in the disintegration theorem, $\zeta_y=\xi_y$ for $\check\mu$-almost all $y\in Y$.
\end{proof}

We shall now consider functions $f:X\to\mathbb{R}$ with a specified amount of regularity, i.e. $f\in\Hol_\omega(X)$ for some modulus $\omega$. We will need a stronger hypothesis on the transfer operator of $S$.
\begin{lemma}\label{lemm:regularity}
Assume that $\chop{L}$ is iteratively bounded with respect to $\omega$.
Then for all $f\in\Hol_\omega(X)$ there is a version of $\xi(f)$ such that for all $y,y'\in Y$ and all $k,n\in\mathbb{N}$:
\[ \lvert \xi_y(f T^k) -\xi_{y'}(f T^k) \rvert \le 2\Hol_\omega(f) \omega(a_{n+k}) + C \Hol_{\omega}(fT^{n+k}\sigma) \omega(d(y,y')).\]
\end{lemma}

\begin{proof}
Set as above $\zeta^n_y(f) = \chop{L}^n(fT^n\sigma)(y)$ for all $y\in Y$. Then $(\zeta^n_y)_{y\in Y}$ is the disintegration of $T^n_*(\sigma_*\check\mu)$, while $(\xi_y)_y$ is the disintegration of $\mu=T^n_*\mu$. There exist $\gamma\in\Gamma_\pi(\mu,\sigma_*\check\mu)$ (actually $\gamma$ is unique, equal to $(\Id,\sigma\pi)_*\mu$), and $\gamma^n := (T^n,T^n)_*\gamma$ is in $\Gamma_\pi(\mu,T^n_*(\sigma_*\check\mu))$. Let $(\eta_y^n)_{y\in Y}$ be the disintegration of $\gamma^n$ with respect to the map $\Delta_\pi \to Y$ sending $(x,x')$ to $\pi(x)=\pi(x')$. Then for $\check\mu$-almost all $y$, $\eta_y^n \in\Gamma(\xi_y,\zeta^n_y)$ and for $\eta_y^n$-almost all $(x,x')$ we have $x=T^n(w)$ and $x'=T^n(w')$ for some $w,w'$ with $\pi(w)=\pi(w')$. Now we get
\begin{align*}
\lvert \xi_y(f T^k) - \zeta^n_y(f T^k) \rvert
  &= \Big\lvert\int f T^k(x) \dd \eta_y^n(x,x') - \int f T^k(x') \dd \eta_y^n(x,x') \Big\rvert \\
  &= \Big\lvert\int (f(x) - f(x')) \dd\big((T^k,T^k)_*\eta_y^n\big)(x,x') \Big\rvert \\
  &\le \Hol_\omega(f) \int  \omega(d(x,x')) \dd\big((T^k,T^k)_*\eta_y^n\big)(x,x') \\
  &\le \Hol_\omega(f) \omega(a_{n+k})
\end{align*}
since for $(T^k,T^k)_*\eta_y^n$-almost all $(x,x')$, $x=T^{k+n}(w)$ and $x'=T^{n+k}(w')$ for some $w,w'$ in the same fiber. 
We then have
\begin{align*}
\lvert \xi_y(f T^k)-\xi_{y'}(f T^k) \rvert 
  &\le \lvert \xi_y(f T^k) - \zeta^n_y(f T^k) \rvert + \lvert \zeta^n_y(f T^k) - \zeta^n_{y'}(f T^k) \rvert \\
  &\qquad+ \lvert \zeta^n_{y'}(f T^k) -\xi_{y'}(f T^k) \rvert \\
  &\le 2\Hol_\omega(f) \omega(a_{n+k}) + \lvert \chop{L}^n(fT^{n+k} \sigma)(y) - \chop{L}^n(fT^{n+k} \sigma)(y') \rvert \\
  &\le  2\Hol_\omega(f) \omega(a_{n+k}) + C \Hol_{\omega}(fT^{n+k}\sigma) \omega(d(y,y')).
\end{align*}
\end{proof}

\begin{theorem}[Regularity of Disintegrations]\label{theo:regularity}
Let $T$ be a $L$-Lipschitz extension of $S$ with $(a_n)_n$-shrinking fibers, and assume that there is a Lipschitz section $\sigma$. Let $\alpha\in (0,1]$ and let $\mu\in\proba_T(X)$, $\check\mu=\pi_*\mu$.
\begin{enumerate}
\item\label{enumi:regularity1} If the transfer operator $\chop{L}$ of $(S,\check\mu)$ is iteratively bounded with respect to $\hol{\alpha}$ and if fibers are exponentially shrinking with ratio $\theta\in (0,1)$, then the disintegration $\xi$ of $\mu$ with respect to $\pi$  is $(\hol{\alpha},\hol{\beta})$-bounded for $\beta = \frac{\alpha}{1-\log L/\log \theta}$.
\item\label{enumi:regularity3} If the transfer operator $\chop{L}$ of $(S,\check\mu)$ is iteratively bounded with respect to $\hol{\alpha}$ and if $(a_n)_n$ is polynomial of degree $d$, then $(\xi_y)_{y\in Y}$ is $(\hol{\alpha},\holl{\alpha d})$-bounded, and moreover the maps
\begin{align*}
\mathcal{D}_k : \Hol_{\alpha}(X) &\to \Hol_{\alpha d\log}(Y) \\
              f  &\mapsto \xi(fT^k)
\end{align*}
have operator norm bounded above by $Ck^{\alpha d}$.
\item\label{enumi:regularity2} If the transfer operator $\chop{L}$ of $(S,\check\mu)$ is iteratively bounded with respect to $\holl{\alpha}$ and if $(a_n)_n$ is exponential, then $(\xi_y)_{y\in Y}$ is  $(\holl{\alpha},\holl{\frac\alpha2})$-bounded.
\end{enumerate}
\end{theorem}
Note that the norm bound in item \ref{enumi:regularity3} is not a specific feature but will be needed later, as the trivial bound of $CL^{\alpha k}$ is much too weak in this case.

Note that Butterley and Melbourne \cite{butterley2017disintegration} (Proposition 6) obtain the better exponent $\beta=\alpha$ in item \ref{enumi:regularity1}, but only in a restricted setting.

\begin{proof}
We apply Lemma \ref{lemm:regularity} three times.

For item \ref{enumi:regularity1}, we take $\omega=\hol{\alpha}$, $k=0$ and get for all $f\in\Hol_\alpha(X)$; $y,y'\in Y$; $k,n\in\mathbb{N}$: 
\[\lvert\xi_y(f)-\xi_{y'}(f) \rvert \lesssim \Hol_\alpha(f)(\theta^{\alpha n} + L^{\alpha n} d(y,y')^\alpha ) 
.\]
Taking $n = \frac{\beta \log d(y,y')}{\alpha\log\theta}$ it comes $\theta^{\alpha n} \simeq L^{\alpha n} d(y,y')^\alpha \simeq d(y,y')^\beta$.

For item \ref{enumi:regularity3} we take $\omega=\hol{\alpha}$ and consider arbitrary $k$ to get the norm estimate. Given $y\neq y'\in Y$ we would like to choose $m=n+k= -c\log d(y,y') +O(1)$ for some small constant $c>0$ to be specified later on. This is possible whenever $k\le-c\log d(y,y')$; in this case we get $L^{\alpha m} \simeq d(y,y')^{-c\alpha\log L}$ and thus
\[\lvert\xi_y(fT^k)-\xi_{y'}(fT^k) \rvert \lesssim \Hol_\alpha(f)\big( \holl{\alpha d}(d(y,y')) + d(y,y')^{\alpha(1-c\log L)} \big).\]
Choosing $c<1/\log L$ ensures the last term above is (much) less than $\holl{\alpha d}(d(y,y'))$. We are left with the case $k>-c\log d(y,y')$, but then $1 \le k/(c\log(1/d(y,y')))$ and
\[\lvert\xi_y(f)-\xi_{y'}(f) \rvert \le \sup f - \inf f \lesssim \Hol_\alpha(f) \lesssim \Hol_\alpha(f) k^{\alpha d} \holl{\alpha d}(d(y,y')).\]

For item \ref{enumi:regularity2}, we take $\omega=\holl{\alpha}$, $k=0$ and get $\omega(a_n) \simeq n^{-\alpha}$. Given $y\neq y'\in Y$ we choose $n \simeq \holl{\frac12}(d(y,y')) \simeq (-\log d(y,y'))^{\frac{1}{2}}$ so that by Proposition \ref{prop:alphalog}, 
\[\Hol_{\alpha\log}(fT^n\sigma) \lesssim (-\log d(y,y'))^{\frac{\alpha}{2}} \Hol_{\alpha\log}(f);\]
then Lemma \ref{lemm:regularity} yields as desired
$\lvert\xi_y(f)-\xi_{y'}(f) \rvert \lesssim \Hol_{\alpha\log}(f)\, \holl{\frac\alpha2}(d(y,y'))$.
\end{proof}

\subsection{Lifting decay of correlations}

Combining Lemma \ref{lemm:CorrBound} with Theorem \ref{theo:regularity}, we can finally lift decay of correlations.
\begin{proof}[Proof of Theorem \ref{theo:maindecay}]
For item \ref{enumi:decay1}, we start from $f\in\Hol_\alpha(X)$ of zero $\mu$-average and $g\in L^1(\mu)$; then $f\in\Hol_{\alpha'}(X)$ for all $\alpha'\le \alpha$, with $\Hol_{\alpha'}(f)\le \Hol_\alpha(f)$. Up to choosing a smaller $\alpha$, the hypothesis on the transfer operator enable us to assume that $\chop{L}$ has a spectral gap on $\Hol_\alpha(Y)$, and is thus $\hol{\alpha}$-iteratively bounded. From Theorem \ref{theo:regularity}, we get that $\xi(fT^k)$ is in $\Hol_\beta(Y)$ for some $\beta\in(0,\alpha)$, with \[\Hol_\beta(\xi(fT^k))\lesssim \Hol_\alpha(f T^k) \lesssim L^{\alpha k} \Hol_\alpha(f)\]
where $L$ is a Lipschitz constant for $T$. Lemma \ref{lemm:CorrBound} then yields for all $k,m\in\mathbb{N}$:
\[\corr{k+m}{\mu}(f,g) \lesssim (1-\delta)^m L^{\alpha k}\Hol_\alpha(f)\lVert \xi(g)\rVert_{L^1(\check\mu)} + \theta^{\alpha k} \Hol_\alpha(f)\lVert g\rVert_{L^1(\mu)}\]
where $\delta$ is the spectral gap of $\chop{L}$ and $\theta$ is the ratio of shrinking (recall $\lVert \xi(g)\rVert_{L^1(\check\mu)}= \lVert g\rVert_{L^1(\mu)}$). Taking sequences $k_n = tn + O(1)$ and $m_n = (1-t)n+O(1)$ summing to $n$ with $t\in(0,1)$ small enough provides exponential decay of $\corr{n}{\mu}(f,g)$.

For \ref{enumi:decay2}, we start again from $f\in\Hol_\alpha(X)$ of zero $\mu$-average and $g\in L^1(\mu)$; by hypothesis $\chop{L}$ is iteratively $\hol{\alpha}$-bounded and $\check\mu$ has polynomial decay of correlations of degree $p$ for $\alpha d\log$-H\"older observables. Theorem \ref{theo:regularity} ensures that $\xi(f T^k)$ is in $\Hol_{\alpha d \log}(Y)$ with norm at most $Ck^{\alpha d}$. Then Lemma \ref{lemm:CorrBound} yields for all $k,m\in\mathbb{N}$:
\[\corr{k+m}{\mu}(f,g) \lesssim \Hol_\alpha(f) \lVert g\rVert_{L^1(\mu)}\Big(\frac{k^{\alpha d}}{m^p} + \frac{1}{k^{\alpha d}}\Big).\]
Given $n$, to optimize over pairs $(k,m)$ such that $k+m=n$, one is led to make both terms of the same order of magnitude, i.e. to take $k\simeq (n-k)^{\frac{p}{2\alpha d}}$. If $p> 2\alpha d$, then $m\ll k$ and thus $k\simeq n$, and we get a polynomial decay of degree $\alpha d$. If $p< 2\alpha d$, then $k \ll m$ and thus $m\simeq n$, $k\simeq m^{\frac{p}{2\alpha d}}$ and we get a polynomial decay of degree $p/2$. If $p= 2\alpha d$, we take $k\simeq m$ both of the same order than $n$ and we get a polynomial decay of correlations of degree $\alpha d = p/2$.

For item \ref{enumi:decay3}, we start from $f\in\Hol_{\alpha \log}(X)$ of zero $\mu$-average and $g\in L^1(\mu)$; by hypothesis $\chop{L}$ is iteratively $\holl{\alpha}$-bounded and $\check\mu$ has polynomial decay of correlations of degree $p$ in $\Hol_{\frac\alpha2\log}(Y)$. Theorem \ref{theo:regularity} ensures that $\xi(fT^k)$ is in $\Hol_{\frac{\alpha}{2}}(Y)$ with norm at most $C\Hol_{\alpha \log}(f T^k)\lesssim k^\alpha  \Hol_{\alpha \log}(f)$ (Proposition \ref{prop:alphalog}). Then Lemma \ref{lemm:CorrBound} yields for all $k,m\in\mathbb{N}$:
\[\corr{k+m}{\mu}(f,g) \lesssim \Hol_{\alpha\log}(f) \lVert g\rVert_{L^1(\mu)}\Big(\frac{k^{\alpha }}{m^p} + \frac{1}{k^{\alpha }}\Big)\]
and, as above, we get $\corr{n}{\mu}(f,g) \lesssim \Hol_{\alpha\log}(f) \lVert g\rVert_{L^1(\mu)} / n^{\min(\alpha, p/2)}$. 
\end{proof}

\section{Proofs of corollaries given in Introduction}\label{sec:proofs}

\begin{proof}[Proof of Corollary \ref{coro:mainPM}]
To prove the first part of item \ref{enumi:PM2}, observe that the unique absolutely continuous invariant measure $\check\mu_S$ of $S_q$ is ergodic and has full support \cite{thaler1980estimates}, and it must thus be the unique physical measure; then Theorem \ref{theo:mainLift} implies that the unique lift $\mu_T$ of $\check\mu_S$ is the unique physical measure of $T$. To obtain the ASIP, consider an $\alpha$-H\"older observable $f$. By Corollary \ref{coro:coboundary}, for some $\gamma>0$ there exists a $\gamma$-H\"older observable $\check f:\mathbb{T}\to\mathbb{R}$ such that $\hat f:=\check f\circ \pi$ differs from $f$ by a coboundary. Then as in Proposition \ref{prop:statistical}, the ASIP for $(T,\mu_T,f)$ follows from the ASIP for $(S_q,\check\mu,\check f)$, which has been proved by Melbourne and Nicol \cite{MN05} (the return time $\mathcal{R}$ of $S_q$ to $[\frac12,1]$ is $n$ on an interval of size $\simeq 1/n^{1+\frac1q}$, so that when $q<\frac12$ we have $\mathcal{R}$ in $L^{2+\delta}$ for some $\delta>0$).
Note that ASIP with better rates have been obtained recently by Cuny, Dedecker, Korepanov and Merlev{\`e}de for intermittent maps \cite{cuny2018rates}; Theorem \ref{theo:stat} enable to lift these rates.

For item \ref{enumi:PM1}, Theorem A of \cite{kloeckner2017optimal} states that $S_q$ satisfies $\UE(\hol{\gamma},\mathrm{CLT},\hol{\ast})$ for all $\gamma>q$ and that the transfer operator of $\gamma$-H\"older potentials have a spectral gap on H\"older spaces of small enough exponent (use \cite{T-K05} for the CLT, see the comment below Corollary F in \cite{kloeckner2017optimal}). Theorem \ref{theo:mainstat}, item \ref{enumi:mainstat1} (where $\beta=1$: $\pi$ is Lipschitz) then ensures that $T$ satisfies $\UE(\hol{\alpha},\mathrm{CLT},\hol{\ast})$ for all $\alpha > q'$. The decay of correlation follows from Theorem \ref{theo:maindecay}.
\end{proof}

\begin{proof}[Proof of Corollary \ref{coro:mainpoly}]
Setting $\alpha'=\alpha d-1>1$, Theorem E of \cite{kloeckner2017optimal} shows that:
\begin{enumerate}
\item the transfer operator associated to a $\alpha'\log$-H\"older potential $\check\varphi$, defined by
\[\chop{L}_{\check\varphi}f(z) = \sum_{z\in S^{-1}(y)} e^{\varphi(y)} f(y) \]
acts on $(\alpha'-1)\log$-H\"older observables; it can be normalized, i.e. up to adding to $\check\varphi$ a constant and a coboundary, one can assume $\chop{L}_{\check\varphi} \one= \one$; and once normalized there is a unique $S$-invariant probability measure $\check\mu_{\check\varphi}$ that is also fixed by the dual operator $\chop{L}_{\check\varphi}^*$,
\item the transfer operator decays polynomially with degree $\alpha'-1$ in the uniform norm for all $u\in\Hol_{(\alpha'-1)\log}(Y)$ such that $\check\mu_{\check\varphi}(u)=0$, i.e. $\lVert \chop{L}_{\check\varphi}^n u\rVert_\infty \lesssim \frac{\Hol_{(\alpha'-1)\log}(u)}{n^{\alpha'-1}} $;
\item when $\alpha'>3/2$, using as above \cite{T-K05}, $\check\mu_{\check\varphi}$ satisfies the CLT for all $(\alpha'-1)\log$-H\"older observables.
\end{enumerate}
While that is not stated in \cite{kloeckner2017optimal}, $\check\mu_{\check\varphi}$ is the unique equilibrium state for $\check\varphi$ (see Ledrappier \cite{ledrappier1974principe} and Walters \cite{walters1975ruelle} -- the statements there are for one-sided subshifts of finite type, but the assumption really used is the existence of a one-sided  generator, which holds here), so that
$S$ satisfies $\UE(\holl{\alpha'},\varnothing)$, and when $\alpha'>3/2$ it also satisfies $\UE(\holl{\alpha'},\mathrm{CLT},\holl{(\alpha'-1)})$.
Theorem \ref{theo:mainstat} enables us to deduce
for $T$ both $\UE(\hol{\alpha};\varnothing)$ when $\alpha>2/d$, and  $\UE(\hol{\alpha};\mathrm{CLT},\hol{\alpha-\frac1d})$ when $\alpha>5/(2d)$ (take $\gamma=\alpha-1/d$, so that $\gamma'=\alpha'-1$). 

Since the transfer operator of a uniformly expanding map with respect to the equilibrium state of a H\"older potential is well-known to have a spectral gap (and thus is iteratively $\hol{\gamma}$-bounded),  Theorem \ref{theo:maindecay} item \ref{enumi:decay2} applies (with the exponent $\gamma=\alpha-1/d$ instead of $\alpha$, and $p=\alpha d-2$), implying a polynomial rate of decay of correlations of degree $\frac{\alpha d}{2}-1$.
\end{proof}

\begin{proof}[Proof of Corollary \ref{coro:mainlog}]
By \cite{kloeckner2017optimal} (see also  \cite{FJ1,FJ2}), $S$ has a unique absolutely continuous measure $\check\mu$, which has polynomial decay of correlations of degree $(\alpha-1)$ for all $(\alpha-1)$-log H\"older observables (in particular, it is ergodic). Let $\mu_T$ be the unique $T$-invariant lift of $\check\mu$ provided by Theorem \ref{theo:mainLift}: then $\mu_T$ is physical, and since $\check\mu$ is also the unique physical measure of $S$, $T$ admits no other physical measure. Better still, its basin of attraction is the inverse image by $\pi$ of the basin of attraction of $\check\mu$ (Corollary \ref{coro:StableLeaf}) thus by Fubini's theorem has full volume.

Moreover, in \cite{kloeckner2017optimal} it is shown that the transfer operator of $S$ for the geometric potential $\varphi_S(y) = -\log \det(DS_y)$ (or any other $\alpha\log$-H\"older potential) has polynomial decay of degree $\gamma=\alpha-1$ in the uniform norm for all $\gamma$-log H\"older observables (implying the Central Limit Theorem as soon as $\alpha>3/2$).
If $T$ has exponentially shrinking fibers, then we can apply the last item of Theorem \ref{theo:maindecay} to $(2\alpha-2)\log$-H\"older observables with $p=\gamma=\alpha-1$ to obtain the desired decay of correlation, and the last item of Theorem \ref{theo:mainstat} to get the Central Limit theorem for $(2\alpha-1)\log$-H\"older observables.
\end{proof}

\begin{proof}[Proof of Corollary \ref{coro:mainSRB}]
We construct $T$ as a Smale DE (``derived from expanding'') example \cite{smale1967differentiable}. As their name indicates, DE examples start from an expanding map of a manifold; we will take $S:\mathbb{T}\to\mathbb{T}$ a uniformly expanding circle map of class $\C^1$ (since we start from a one-dimensional base map, this kind of example can be called a ``solenoidal'' example: the attractor will topologically be a solenoid). For some $\lambda>1$, we have $S'(y) \ge \lambda$ for all $y\in\mathbb{T}$; we then take a skew product 
\begin{align*}
T:\mathbb{T}\times D^2 &\to \mathbb{T}\times D^2 \\
  (y,z) &\mapsto (S(y),R(y,z))
\end{align*}
where $D^2$ is the open unit disk of $\mathbb{R}^2$, $R$ is smooth and chosen so that 
\begin{itemize}
\item $T$ is a diffeomorphism onto its image (i.e. for all $y\in\mathbb{T}$, $R(y,\cdot):D^2\to D^2$ is a diffeomorphism onto its image  and whenever $y,y'\in\mathbb{T}$ have the same image under $S$, $R(y\cdot)$ and $R(y',\cdot)$ are disjoint),
\item we moreover assume that whenever $S(y)=S(y')$, the images of $R(y,\cdot)$ and $R(y',\cdot)$ have disjoint closures; in particular the closure in $\mathbb{T}\times D^2$ of the image $\mathrm{Im}(T)$ is compact,
\item $\lVert D_z R(y,z) \rVert \le \lambda^{-1}$, in particular $T$ is an extension of $S$ with exponentially shrinking fibers.
\end{itemize}
We identify $\mathbb{T}\times D^2$ with an open subset $U$ of $\mathbb{R}^3$ (e.g. a solid torus of revolution, with the angle of cylindrical coordinates corresponding to the $y$ variable). By assumption, $\Lambda=\bigcap_n T^n(U)$ is a compact attractor, and one checks easily that the restriction of the projection $\pi$ to $\Lambda$ is still onto $\mathbb{T}$; we denote this restriction by the same letter $\pi$.

We first check that $\Lambda$ is uniformly hyperbolic (this argument is classical and can be skipped by the experienced reader). The stable bundle is trivially constructed over the whole of $U$ as $E^s_x = \{0\} \times T_z D^2$ (where $x=(y,z)$), and the main point is to find a transversal bundle $E^u$ that is $T$-invariant. We consider the space of all continuous $1$-dimensional sub-bundles $E\subset T_{\Lambda}U$ transversal to $E^s$; such a bundle is parametrized by a field $(L^E_x)_{x=(y,z)\in\Lambda}$ of linear maps $T_y\mathbb{T} \to T_zD^2$, simply setting $E_{x}=\{(u,L_x(u))\in T_xU : u\in T_y \mathbb{T}\}$ (i.e. $L_x(u)$ is the unique $v\in T_zD^2$ such that $u+v\in E_x$), and we obtain a complete metric by using the operator norm: $d(E,F)= \max_x \lVert L^E_x-L^F_x \rVert$. Now the facts that $S$ is at least $\lambda$-expanding and that $\lVert D_zR\rVert \le \lambda^{-1}$ ensure that $T$ acts on this space of bundles as a contraction: writing $x'=(y',z')=T^{-1}(x)$, the definition $(T_*E)_x = DT_{x'}(E_{x'})$ translates as $(T_*L)_x(u) = D_z R_{x'} \circ L_{x'} (DS_{x'}^{-1}(u))$, so that $d(T_*E,T_*F) \le \lambda^{-2} d(E,F)$. There is thus a unique $T$-invariant continuous sub-bundle transverse to $E^s$. Up to changing the Riemannian metric, we can make it coincide on $E^u$ with the pull-back of the metric of $\mathbb{T}$; then $DT$ is at least $\lambda$-expanding along $E^u$ in this metric, so that $\Lambda$ is uniformly hyperbolic.

Now the usual theory ensures we have an unstable lamination $W^u$ of $\Lambda$ (the stable foliation is trivial, its connected components of leaves being the vertical slices $\{y\}\times D^2$), and the definition of an SRB makes sense. We shall use the following lemma.
\begin{lemma}
If $\mu$ is an SRB measure of $T=(S,R)$ a uniformly hyperbolic skew product, then the projection $\check\mu$ of $\mu$ to the first factor is absolutely continuous.
\end{lemma}
\begin{proof}
Consider a small open set $V\subset \mathbb{T}$, and partition $\pi^{-1}(V)\subset \Lambda$ into small enough subsets $V_1,\dots, V_k$ such that each $V_k$ is given a product structure by $W^s,W^u$. For each $i\in\{1,\dots,k\}$, let $\mu^i = \mu_{|V_i}$ and write its disintegration with respect to the projection on the stable direction as $\mu^i = \int \mu^i_L \dd\nu^i(L)$. If $\mu$ is SRB, there are positive integrable functions $f^i_L$ such that $\dd\mu^i_L = f^i_L \dvol$ where $\dvol$ is the volume (i.e. Lebesgue measure) on $\mathbb{T}$. Then for any continuous $g:\mathbb{T}\to \mathbb{R}$:
\[\check\mu_{|V}(g) = \sum_{i=1}^k \iint g(y) f^i_L(y) \dvol(y) \dd\nu^i(L) = \int g(y) \Big(\sum_{i=1}^k \int f^i_L(y) \dd\nu^i(L)\Big) \dvol(y) \]
is absolutely continuous.
\end{proof}
Note that we did not use invariance of $\mu$ and that the converse of this Lemma is not obvious: there are (non necessarily invariant) measures that project to the Lebesgue measure without having absolutely continuous disintegrations.

The work of Campbell and Quas \cite{campbell2001generic} shows that taking $S$ generic, we can assume it has a unique physical measure $\check\mu$, with full basin of attraction, but that is singular to $\dvol$ (and thus $S$ has no Acip). Then its lift $\mu$ is a $T$-invariant measure that is physical, with full basin of attraction. Moreover $T$ has no SRB measure, since it would have an absolutely continuous projection.
\end{proof}

\bibliographystyle{alpha}
\bibliography{extensions}

\begin{thebibliography}{CDKM19}

\bibitem[ABV00]{alves2000SRB}
Jos\'{e}~F. Alves, Christian Bonatti, and Marcelo Viana.
\newblock S{RB} measures for partially hyperbolic systems whose central
  direction is mostly expanding.
\newblock {\em Invent. Math.}, 140(2):351--398, 2000.

\bibitem[ADLP17]{alves2017SRB}
Jos\'{e}~F. Alves, Carla~L. Dias, Stefano Luzzatto, and Vilton Pinheiro.
\newblock S{RB} measures for partially hyperbolic systems whose central
  direction is weakly expanding.
\newblock {\em J. Eur. Math. Soc. (JEMS)}, 19(10):2911--2946, 2017.

\bibitem[AGP14]{araujo2014decay}
Vitor Ara\'{u}jo, Stefano Galatolo, and Maria~Jos\'{e} Pacifico.
\newblock Decay of correlations for maps with uniformly contracting fibers and
  logarithm law for singular hyperbolic attractors.
\newblock {\em Math. Z.}, 276(3-4):1001--1048, 2014.

\bibitem[AGY06]{avila2006exponential}
Artur Avila, S\'ebastien Gou\"ezel, and Jean-Christophe Yoccoz.
\newblock Exponential mixing for the {T}eichm\"uller flow.
\newblock {\em Publ. Math. Inst. Hautes \'Etudes Sci.}, 104:143--211, 2006.

\bibitem[Alv15]{alves2015srb}
Jos{\'e}~F Alves.
\newblock S{RB} measures for partially hyperbolic attractors, 2015.

\bibitem[AMV15]{araujo2015rapid}
V.~Ara\'{u}jo, I.~Melbourne, and P.~Varandas.
\newblock Rapid mixing for the {L}orenz attractor and statistical limit laws
  for their time-1 maps.
\newblock {\em Comm. Math. Phys.}, 340(3):901--938, 2015.

\bibitem[APPV09]{araujo2009singular}
V.~Araujo, M.~J. Pacifico, E.~R. Pujals, and M.~Viana.
\newblock Singular-hyperbolic attractors are chaotic.
\newblock {\em Trans. Amer. Math. Soc.}, 361(5):2431--2485, 2009.

\bibitem[AV09]{abdenur2009flavors}
Fl{\'a}vio Abdenur and Marcelo Viana.
\newblock Flavors of partial hyperbolicity.
\newblock 2009.

\bibitem[BM17]{butterley2017disintegration}
Oliver Butterley and Ian Melbourne.
\newblock Disintegration of invariant measures for hyperbolic skew products.
\newblock {\em Israel J. Math.}, 219(1):171--188, 2017.

\bibitem[Bow08]{bowen1975equilibrium}
Rufus Bowen.
\newblock {\em Equilibrium states and the ergodic theory of {A}nosov
  diffeomorphisms}, volume 470 of {\em Lecture Notes in Mathematics}.
\newblock Springer-Verlag, Berlin, revised edition, 2008.
\newblock With a preface by David Ruelle, Edited by Jean-Ren\'{e} Chazottes.

\bibitem[CDKM19]{cuny2018rates}
C~Cuny, J~Dedecker, A~Korepanov, and Florence Merlev{\`e}de.
\newblock Rates in almost sure invariance principle for slowly mixing dynamical
  systems.
\newblock {\em Ergodic Theory Dynam. Systems}, 2019.
\newblock arXiv:1801.05335.

\bibitem[CE11]{catsigeras2011srb}
Eleonora Catsigeras and Heber Enrich.
\newblock S{RB}-like measures for {$C^0$} dynamics.
\newblock {\em Bull. Pol. Acad. Sci. Math.}, 59(2):151--164, 2011.

\bibitem[CN17]{castro2017statistical}
Armando Castro and Te\'{o}filo Nascimento.
\newblock Statistical properties of the maximal entropy measure for partially
  hyperbolic attractors.
\newblock {\em Ergodic Theory Dynam. Systems}, 37(4):1060--1101, 2017.

\bibitem[CP15]{crovisier2015introduction}
Sylvain Crovisier and Rafael Potrie.
\newblock Introduction to partially hyperbolic dynamics.
\newblock {\em School on Dynamical Systems, ICTP, Trieste}, 3, 2015.

\bibitem[CPZ19]{climenhaga2018equilibrium}
Vaughn Climenhaga, Yakov Pesin, and Agnieszka Zelerowicz.
\newblock Equilibrium states in dynamical systems via geometric measure theory.
\newblock {\em Bull. Amer. Math. Soc.}, 56:569--610, 2019.
\newblock arXiv:1803.10374.

\bibitem[CQ01]{campbell2001generic}
James~T. Campbell and Anthony~N. Quas.
\newblock A generic {$C^1$} expanding map has a singular {S}-{R}-{B} measure.
\newblock {\em Comm. Math. Phys.}, 221(2):335--349, 2001.

\bibitem[CV13]{CV}
A.~Castro and P.~Varandas.
\newblock Equilibrium states for non-uniformly expanding maps: decay of
  correlations and strong stability.
\newblock {\em Ann. Inst. H. Poincar\'e Anal. Non Lin\'eaire}, 30(2):225--249,
  2013.

\bibitem[DHRS09]{diaz2009destroying}
L.~J. D\'{i}az, V.~Horita, I.~Rios, and M.~Sambarino.
\newblock Destroying horseshoes via heterodimensional cycles: generating
  bifurcations inside homoclinic classes.
\newblock {\em Ergodic Theory Dynam. Systems}, 29(2):433--474, 2009.

\bibitem[FJ01a]{FJ2}
Aihua Fan and Yunping Jiang.
\newblock On {R}uelle-{P}erron-{F}robenius operators. {I}. {R}uelle theorem.
\newblock {\em Comm. Math. Phys.}, 223(1):125--141, 2001.

\bibitem[FJ01b]{FJ1}
Aihua Fan and Yunping Jiang.
\newblock On {R}uelle-{P}erron-{F}robenius operators. {II}. {C}onvergence
  speeds.
\newblock {\em Comm. Math. Phys.}, 223(1):143--159, 2001.

\bibitem[Gal18]{galatolo2018quantitative}
Stefano Galatolo.
\newblock Quantitative statistical stability and speed of convergence to
  equilibrium for partially hyperbolic skew products.
\newblock {\em J. Éc. polytech. Math.}, 5:377--405, 2018.
\newblock arXiv preprint arXiv:1702.05996.

\bibitem[GL20]{galatolo2015spectral}
Stefano Galatolo and Rafael Lucena.
\newblock Spectral gap and quantitative statistical stability for systems with
  contracting fibers and lorenz like maps.
\newblock {\em Discrete Contin. Dyn. Syst.}, 40(3):1309--1360, 2020.
\newblock arXiv:1507.08191.

\bibitem[GNP18]{galatolo2018decay}
Stefano Galatolo, Isaia Nisoli, and Maria~Jose Pacifico.
\newblock Decay of correlations, quantitative recurrence and logarithm law for
  contracting lorenz attractors.
\newblock {\em Journal of Statistical Physics}, 170(5):862--882, 2018.

\bibitem[Gou10]{gouezel2010almost}
S{\'e}bastien Gou{\"e}zel.
\newblock {A}lmost {S}ure {I}nvariance {P}rinciple for dynamical systems by
  spectral methods.
\newblock {\em The Annals of Probability}, 38(4):1639--1671, 2010.

\bibitem[GP10]{galatolo2010lorenz}
Stefano Galatolo and Maria~Jos\'{e} Pacifico.
\newblock Lorenz-like flows: exponential decay of correlations for the
  {P}oincar\'{e} map, logarithm law, quantitative recurrence.
\newblock {\em Ergodic Theory Dynam. Systems}, 30(6):1703--1737, 2010.

\bibitem[Gui15]{guiheneuf2015dynamical}
Pierre-Antoine Guih\'eneuf.
\newblock Dynamical properties of spatial discretizations of a generic
  homeomorphism.
\newblock {\em Ergodic Theory Dynam. Systems}, 35(5):1474--1523, 2015.

\bibitem[Klo20]{kloeckner2017optimal}
Beno{\^\i}t Kloeckner.
\newblock An optimal transportation approach to the decay of correlations for
  non-uniformly expanding maps.
\newblock {\em Ergodic Theory Dynam. Systems}, 40(3):714--750, 2020.
\newblock arXiv:1711.08052.

\bibitem[KRN65]{kuratowski1965general}
K.~Kuratowski and C.~Ryll-Nardzewski.
\newblock A general theorem on selectors.
\newblock {\em Bull. Acad. Polon. Sci. S\'{e}r. Sci. Math. Astronom. Phys.},
  13:397--403, 1965.

\bibitem[Led74]{ledrappier1974principe}
Fran{\c{c}}ois Ledrappier.
\newblock Principe variationnel et systemes dynamiques symboliques.
\newblock {\em Zeitschrift f{\"u}r Wahrscheinlichkeitstheorie und Verwandte
  Gebiete}, 30(3):185--202, 1974.

\bibitem[Led84]{ledrappier1984proprietes}
Fran{\c{c}}ois Ledrappier.
\newblock Propri{\'e}t{\'e}s ergodiques des mesures de sinai.
\newblock {\em Publications Math{\'e}matiques de l'Institut des Hautes
  {\'E}tudes Scientifiques}, 59(1):163--188, 1984.

\bibitem[LOR11]{leplaideur2011equilibrium}
R.~Leplaideur, K.~Oliveira, and I.~Rios.
\newblock Equilibrium states for partially hyperbolic horseshoes.
\newblock {\em Ergodic Theory Dynam. Systems}, 31(1):179--195, 2011.

\bibitem[LW77]{ledrappier1977relativised}
Fran\c{c}ois Ledrappier and Peter Walters.
\newblock A relativised variational principle for continuous transformations.
\newblock {\em J. London Math. Soc. (2)}, 16(3):568--576, 1977.

\bibitem[MN05]{MN05}
Ian Melbourne and Matthew Nicol.
\newblock Almost sure invariance principle for nonuniformly hyperbolic systems.
\newblock {\em Comm. Math. Phys.}, 260(1):131--146, 2005.

\bibitem[MT02]{melbourne2002central}
Ian Melbourne and Andrei T{\"o}r{\"o}k.
\newblock {C}entral {L}imit {T}heorems and {I}nvariance {P}rinciples for
  time-one maps of hyperbolic flows.
\newblock {\em Communications in Mathematical Physics}, 229(1):57--71, 2002.

\bibitem[Roh52]{rohlin1952fundamental}
V.~A. Rohlin.
\newblock On the fundamental ideas of measure theory.
\newblock {\em Amer. Math. Soc. Translation}, 1952(71):55, 1952.

\bibitem[RS17]{ramos2017equilibrium}
V.~Ramos and J.~Siqueira.
\newblock On equilibrium states for partially hyperbolic horseshoes: uniqueness
  and statistical properties.
\newblock {\em Bull. Braz. Math. Soc. (N.S.)}, 48(3):347--375, 2017.

\bibitem[Sim12]{simmons2012conditional}
David Simmons.
\newblock Conditional measures and conditional expectation; {R}ohlin's
  disintegration theorem.
\newblock {\em Discrete Contin. Dyn. Syst.}, 32(7):2565--2582, 2012.

\bibitem[Sma67]{smale1967differentiable}
S.~Smale.
\newblock Differentiable dynamical systems.
\newblock {\em Bull. Amer. Math. Soc.}, 73:747--817, 1967.

\bibitem[Tha80]{thaler1980estimates}
Maximilian Thaler.
\newblock Estimates of the invariant densities of endomorphisms with
  indifferent fixed points.
\newblock {\em Israel J. Math.}, 37(4):303--314, 1980.

\bibitem[TK05]{T-K05}
Marta Tyran-Kami\'nska.
\newblock An invariance principle for maps with polynomial decay of
  correlations.
\newblock {\em Comm. Math. Phys.}, 260(1):1--15, 2005.

\bibitem[Vil09]{villani2009oldandnew}
C{\'e}dric Villani.
\newblock {\em Optimal transport}, volume 338 of {\em Grundlehren der
  Mathematischen Wissenschaften [Fundamental Principles of Mathematical
  Sciences]}.
\newblock Springer-Verlag, Berlin, 2009.
\newblock Old and new.

\bibitem[Wal75]{walters1975ruelle}
Peter Walters.
\newblock Ruelle's operator theorem and {$g$}-measures.
\newblock {\em Trans. Amer. Math. Soc.}, 214:375--387, 1975.

\bibitem[Wal82]{walters1982introduction}
P.~Walters.
\newblock {\em An Introduction to Ergodic Theory}, volume~79 of {\em Graduate
  Texts in Mathematics}.
\newblock Springer Verlag, 1982.

\bibitem[You98]{young98}
Lai-Sang Young.
\newblock Statistical properties of dynamical systems with some hyperbolicity.
\newblock {\em Ann. of Math. (2)}, 147(3):585--650, 1998.

\bibitem[You02]{young2002srb}
Lai-Sang Young.
\newblock What are {SRB} measures, and which dynamical systems have them?
\newblock {\em Journal of Statistical Physics}, 108(5-6):733--754, 2002.

\end{thebibliography}
\end{document}